\pdfoutput=1
\documentclass{amsart}
\usepackage{booktabs}
\usepackage{amssymb}
\makeindex

\usepackage{hyperref}
\hypersetup{
    colorlinks,%
    citecolor=blue,%
    filecolor=blue,%
    linkcolor=blue,%
    urlcolor=blue
}

\usepackage{longtable}
 \usepackage[all,cmtip]{xy}
\usepackage{tikz}
\usepackage{todonotes}
\usepackage[backrefs,lite]{amsrefs}
\usepackage{enumerate}
\usepackage{verbatim}
\usepackage{multirow}
\usepackage{pgf,tikz}
\usetikzlibrary{arrows}
\usepackage{tkz-graph}
\usepackage{mathtools}
\usepackage[normalem]{ulem}
\usepackage{enumitem}
\usepackage{flafter}
\usepackage{caption}
\CompileMatrices
\newtheorem*{introthm}{Theorem}
\newtheorem*{introprop}{Proposition}

\newtheorem{theorem}{Theorem}[section]
\newtheorem{lemma}[theorem]{Lemma}

\newtheorem{proposition}[theorem]{Proposition}
\newtheorem{corollary}[theorem]{Corollary}
\theoremstyle{definition}
\newtheorem{definition}[theorem]{Definition}
\newtheorem{example}[theorem]{Example}

\newtheorem{remark}[theorem]{Remark}

\newtheorem{notation}{Notation}
\newtheorem{assumption}{Assumption}

\def\cc{{\mathbb C}}

\def\zz{{\mathbb Z}}
\def\rr{{\mathbb R}}

\def\qq{{\mathbb Q}}
\def\pp{{\mathbb P}}

\def\Osh{{\mathcal O}}

\def\Conv{\operatorname{Conv}} 
\def\Cone{\operatorname{Cone}}

\def\Conv{\operatorname{Conv}} 
\def\Cl{\operatorname{Cl}}

\begin{document}
\title[]{Calabi-Yau complete intersections  associated to good pairs of generalized nef partitions
}

\author{Michela Artebani}
\address{
Departamento de Matem\'atica, \newline
Universidad de Concepci\'on, \newline
Casilla 160-C,
Concepci\'on, Chile}
\email{martebani@udec.cl}
\author{Paola Comparin}
\address{
Departamento de Matem\'atica y Estad\'istica, \newline
Universidad de La Frontera, \newline
Temuco, Chile}
\email{paola.comparin@ufrontera.cl} 

\author{Robin Guilbot}
\address{D\'epartement de Math\'ematiques, \newline \'Ecole d'\'Economie de Toulouse, \newline Toulouse, France
}
\email{robin.guilbot@ut-capitole.fr}

\thanks{The first author has been partially  supported by Proyectos Fondecyt Regular 1160897 and 1211708,
the second author has been partially 
supported by Proyecto Fondecyt  N. 1240360 and is member of INdAM-GNSAGA}

\begin{abstract} 
We introduce the notion of good pair of generalized nef partitions to describe Calabi-Yau complete intersections in $\qq$-Fano toric varieties
whose equations do not necessarily have maximal Newton polytopes. Moreover, we define a natural duality between them which generalizes Batyrev-Borisov mirror duality and allows to define a generalization of Berglund-H\"ubsch-Krawitz duality to quasismooth complete intersections.
\end{abstract}

\maketitle
\tableofcontents

\section*{Introduction}
The foundational toric constructions of Calabi-Yau varieties begin with the observation that anticanonical hypersurfaces in Gorenstein Fano toric varieties have trivial canonical class. Batyrev's reflexive polytope construction \cite{Bat} and its extension to complete intersections by Batyrev-Borisov via nef-partitions \cite{BB} produced vast classes of Calabi-Yau hypersurfaces and complete intersections admitting mirror partners described purely in terms of dual polyhedral data. These constructions not only supplied many new examples, but also gave a precise and computable form of mirror symmetry: mirror pairs arise from duality of reflexive polytopes and, in the complete-intersection setting, from dual nef-partitions. In a different but related direction, the Berglund-H\"ubsch-Krawitz construction \cite{BH, Kr}, originating in Landau-Ginzburg models and invertible polynomials, produces mirror partners of hypersurfaces in fake weighted projective spaces by transposition of exponent matrices, together with an orbifolding procedure.  
In \cite{ACG} we developed a geometric framework to deal with families of Calabi-Yau hypersurfaces  
in toric $\qq$-Fano varieties  whose Newton polytope is not necessarily maximal.
More precisely, we defined the notion of good pair of polytopes to encode this type of families and we defined a duality between such pairs which  generalizes both the Batyrev construction and the Berglund-H\"ubsch-Krawitz  construction.
 The purpose of this paper is to generalize this framework to complete intersection Calabi-Yau varieties.
 
As a first step, we provide a generalization of the notion of nef-partition, to include the case when the ambient space is $\mathbb Q$-Fano and the nef divisors associated to the partition are $\qq$-Cartier. A generalized nef partition is denoted by $\Pi(\Delta)=\{\Delta_1,\ldots,\Delta_s\}$, where  $\Delta$ is the Minkowski sum of $\Delta_1,\ldots,\Delta_s$.
Afterwards, we define good pairs of generalized nef partitions as pairs of nested generalized nef partitions $\mathcal P=(\Pi(\Delta_1),\Pi(\Delta_2))$, where $\Delta_1\subseteq\Delta_2$ form a good pair \cite{ACG}, i.e. $\Delta_1$ and $(\Delta_2)^\circ$ are canonical polytopes.
This framework is related to the construction of Calabi-Yau varieties as follows. 
Given a good pair of  
generalized nef partitions
 $(\Pi(\Delta^1),\Pi(\Delta^2))$, where  $\Pi(\Delta^1)=\{\Delta^1_1,\ldots,\Delta^1_s\}$ and $\Pi(\Delta^2)=\{\Delta^2_1,\ldots,\Delta^2_s\}$,
let $Z=Z_{\Delta^2}$ be the toric variety defined by the normal fan to $\Delta^2$, which is $\qq$-Fano with canonical singularities.
The generalized nef partition $\Pi(\Delta^2)$ provides a decomposition  
$-K_Z=\sum_{i=1}^s N_i$, where $N_i$ are toric $\qq$-Cartier nef divisors with associated polytopes $\Delta^2_1,\dots,\Delta_s^2$. 
For each $i$, the polytope $\Delta^1_i \subseteq\Delta_i^2$  identifies a subspace $\mathcal F(\Delta^2_i)$ of the complete linear system associated to $N_i$. We  call $g_i\in \mathcal R(Z)$ the equation of a general element in the subspace.

We define a good pair of generalized nef partitions to be {\em quasismooth} if the subvariety of $Z$ defined by $g_1=\dots=g_s=0$ is a quasismooth complete intersection, {\em irreducible} if $\Pi(\Delta_1)$ and $\Pi(\Delta_2)$ can not be decomposed as a direct sum (see \S\ref{irr}).

\begin{introthm}[see Theorem \ref{CY}]\label{thm_main1}
If $\mathcal P=(\Pi(\Delta^1),\Pi(\Delta^2))$ is an irreducible and quasismooth good pair of 
generalized nef partitions, then 
$g_1,\dots,g_s$ 
 define a family of Calabi-Yau varieties $\mathcal F_{\mathcal P}$ in the $\qq$-Fano toric variety $Z_{\Delta^2}$.
\end{introthm}

The above construction allows to define naturally a duality between families of Calabi-Yau complete intersections, i.e. given a good pair of generalized nef partitions  $\mathcal P=(\Pi(\Delta^1),\Pi(\Delta^2))$, one can consider the dual pair  $\mathcal P^*=(\Pi(\nabla^2),\Pi(\nabla^1))$, which still defines a family of Calabi-Yau varieties provided $\mathcal P^*$ is quasismooth. Observe that this is essentially Batyrev-Borisov duality when $\Delta_1=\Delta_2$ is reflexive.

In case the the good pair $\mathcal P$ is a Delsarte pair (see Definition \ref{delsarte}), 
the ambient spaces for both families $\mathcal F_\mathcal P$ and $\mathcal F_{\mathcal P^*}$ are fake weighted projectives spaces.
Thus, if quasismoothness holds on both sides, this is a generalization of the Berglund-H\"ubsch-Krawitz construction.
We apply the duality of good pairs to several examples and for some of them, such as an example by Libgober-Teitelbaum \cite{LT93} and a family of Schoen's type \cite{Schoen, HSS}, we recover mirror families already found in the literature.  
In the special case of codimension two K3 surfaces in weighted projective spaces we prove the following.

\begin{introprop}[see Proposition \ref{prop-codim2}]
  Let $\mathcal F_{\mathcal P}$ be a family of  K3 surfaces in a five dimensional weighted projective space  associated to a quasismooth good pair $\mathcal P$ of generalized nef partitions of Delsarte type. 
Then $\mathcal P^*$ is quasismooth and thus $\mathcal F_{\mathcal P^*}$ is a family of codimension two K3 surfaces in a fake weighted projective space.
\end{introprop}

Quasismoothness is a key property throughout the paper, since it allows to control the singularities of the Calabi-Yau varieties. In section \S\ref{2.4} we provide a combinatorial characterization for quasismooth complete intersection based on the Cayley trick and results in \cite{ACGqs}.
Although quasismoothness is in general not preserved by the duality of good pairs (see Example \ref{counterexample}) we expect that it could be true in a more general setting and we intend to explore it in a future work.

The content of the paper can be summarized as follows.
Section \ref{sec-background} contains the preliminaries about complete intersections in toric varieties and the quasismoothness condition. In section \ref{sec-gnp} we introduce the notion of generalized nef partitions and the duality between them, while in section \ref{sec_gp} we define good pairs of generalized nef partitions and prove the main theorem. In particular, we explain the relation between polytopes and complete intersections and give several examples. Section \ref{codim2} contains the proof of the Proposition.
Finally, the Appendix contains the link to a webpage containing the ancillary MAGMA \cite{Magma} code and a short description of the main functions in it.

\section{Toric background}\label{sec-background}

We start introducing the basic language for complete intersections in toric varieties, see \cite{Khov,CLS, BB, antican, mavlyutov}. 
We will always work over the field of complex numbers.
\subsection{Toric setting}\label{sec_toric}
Let $N\cong \zz^n$ be a lattice, $M={\rm Hom}(N,\zz)$ be its dual and 
$N_{\qq}, M_{\qq}$ be their $\qq$-extensions. We will denote by $\langle\,,\,\rangle:M_\qq\times N_{\qq}\to \zz$ the natural bilinear pairing.
In what follows $Z=Z_{\Sigma}$ denotes the toric variety associated to a fan $\Sigma\subseteq N_{\qq}$. 
Fixed an order for the one dimensional cones $\rho_1,\dots,\rho_r$ of $\Sigma$, let $P:\zz^r\to N$ be the homomorphism 
which associates to $e_i$ the primitive generator $n_i\in N$ of $\rho_i$ and let $P^T$ be its dual. We recall that there is an exact sequence
\[
\xymatrix{
0\ar[r]& M\ar[r]^{P^{T}}& \zz^r\ar[r]^Q& \Cl(Z)\ar[r]& 0,
}
\]
where $Q(e_i)$ is the class of the prime toric divisor $D_i$ associated to $\rho_i$.
The {\em homogeneous coordinate ring} of $Z$ is the $\Cl(Z)$-graded polynomial ring 
\[
\mathcal R(Z)=\cc[x_1,\dots,x_r],
\]
where $\deg(x_i):=Q(e_i)$.
The homomorphism $p:\mathbb T^r\to \mathbb T^n$ between tori induced by $P^T$ extends to a morphism $\hat p:\hat Z\to Z$,  where $\hat Z$ is the toric variety whose fan is given by the cones of the positive orthant in $\qq^r$ whose image by $P_{\qq}$ is contained in a cone of $\Sigma$ \cite{CLS}.
The morphism $\hat p$, called {\em characteristic morphism} of $Z$, is a good quotient for the action of $H:={\rm Spec}(\cc[\Cl(Z)])$. We recall that $\hat Z$ is an open subset of $\bar Z={\rm Spec}\,\mathcal R(Z)$ whose complement has codimension at least two. 
The ideal in $\mathcal R(Z)$ defining $\bar Z\backslash \hat Z$ is the {\em irrelevant ideal} of $Z$.

\subsection{Laurent systems and their homogenization}
\label{subsec-polytopes}
Given a Laurent polynomial $f=\sum_{v\in M} c_vT^v\in \cc[M] \cong\cc[T_1^{\pm },\dots, T_n^{\pm}] $ we denote by $V(f)\subseteq \mathbb T^n$ its zero set and by $\Delta(f)$ its {\em Newton polytope} in $M_{\qq}\cong \qq^n$:
\[
\Delta(f)={\rm Conv}(v\in M: c_v\not=0).
\]
The {\em $\Sigma$-homogenization} of $f$ 
is the polynomial $g\in \mathcal R(Z)$ defined by
\[
(p^*f)(x_1,\dots,x_r)=x^\nu g(x_1,\dots, x_r),
\]
where $x^\nu\in \cc[x_1^{\pm },\dots, x_r^{\pm}]$ is a Laurent monomial and $g$ is a $\Cl(Z)$-homogeneous polynomial coprime to each of the variables $x_i$.

\begin{lemma} The $\Sigma$-homogenization $g$ of $f$ is a defining polynomial for the divisor $D$ which is the closure of $V(f)$ in $Z_{\Sigma}$.  
Moreover  $[D]=[\sum_{i=1}^r a_iD_i]\in \Cl(Z)$,
where $D_i$ is the toric prime divisor associated to $\rho_i$ and 
\[
(a_1,\dots,a_r):=(-\min_{u\in \Delta(f)}\langle u,n_1\rangle,\dots, -\min_{u\in \Delta(f)}\langle u_,n_r\rangle)=-\nu.
\]
\end{lemma}
\begin{proof}
See \cite[Lemma 3.3]{antican}.
\end{proof}

Conversely, given a $\Cl(Z)$-homogeneous $g\in \mathcal R(Z)$ of degree 
$w\in \Cl(Z)$ and a toric divisor $D=\sum_{i=1}^r a_iD_i$ whose class is $w$, the $\Sigma$-{\em dehomogenization of} $g$ {\em with respect to} $D$ is the Laurent polynomial $f\in \cc[T_1^{\pm },\dots, T_n^{\pm}]$ such that 
$p^*f=x^\nu g$, where 
$\nu=-(a_1,\dots, a_r)$. Observe that the Newton polytopes of different dehomogenizations of $g$ differ by a translation by a vector in $\zz^n$.

\begin{notation}\label{not1}
\ \\
\vspace{-0.4cm}

\begin{itemize}[leftmargin=2pt]
\item 
Let $\Delta(D)$ be the polytope associated to $D$:
\[
\Delta(D):=\{m\in M_{\qq}: \langle m,n_i\rangle\geq -a_i,\ \forall i=1,\dots,r\}\subseteq M_{\qq},
\]
which is the Newton polytope of the 
$\Sigma$-dehomogenization with respect to $D$ of a general homogeneous $g\in \mathcal R(Z)$ of degree $w=[\sum_{i=1}^r a_iD_i]$.  Observe that the lattice points of $\Delta(D)$ correspond to the monomials of $\mathcal R(Z)$ of degree $w$.
\item We denote by $\mathcal F(\Delta(D))=|D|$ the complete linear system of $Z$ associated to $D$ (i.e. the family of all effective divisors linearly equivalent to $D$).
\item Given $g\in \mathcal R(Z)$ of degree $w=[D]$,
let $\Delta(D,g)\subseteq \Delta(D)$ be the Newton polytope of the $\Sigma$-dehomogenization of $g$ with respect to $D$.

\item
Given a polytope $\Theta\subseteq \Delta(D)$ we denote by
 $\mathcal F(\Theta)$ the subsystem of $|D|$ associated to $\Theta$, i.e. the set of effective divisors obtained as zero sets of linear combinations of the monomials in $\mathcal R(Z)$ corresponding to the lattice points of $\Theta$.
\end{itemize}
\end{notation}

More generally, given a set $f=(f_1,\dots, f_s)$ of
Laurent polynomials in $\cc[T_1^{\pm },\dots, T_n^{\pm}]$, we define the {\em Newton polytope of} $f=(f_1,\dots, f_s)$
as the Minkowski sum 
\[
\Delta(f):=\Delta(f_1)+\cdots+\Delta(f_s).
\]
The {\em $\Sigma$-homogenization} of $f=(f_1,\dots,f_s)$ is the set $g=(g_1,\dots,g_s)$,
where $g_i$ is the $\Sigma$-homogenization of $f_i$. 
Similarly we define the $\Sigma$-{\em dehomogenization of} a set $g=(g_1,\dots,g_s)$ of homogeneous elements in $\mathcal R(Z)$  {\em with respect to} a collection of toric divisors 
  $D_1,\dots,D_s$  such that $[D_i]=\deg(g_i)$.

 Let $\overline X$ be the zero set of the set of the polynomial system $g=(g_1,\dots,g_s)$ in $\bar Z$ and 
$X=\overline{ V(f_1)}\cap\dots\cap \overline{V(f_s)}\subseteq Z$. We have a commutative diagram

\[
\xymatrix{
\hat X\ar@{^{(}->}[r]\ar[d]_{\hat p_{|\hat X}} & \hat Z\ar[d]^{\hat p}\\
X\ar@{^{(}->}[r]& Z
}
\]
where $\hat X=\overline X\cap \hat Z$ and $\hat p_{|\hat X}:\hat X\to X$ is a good quotient for the action of $H$ on $\hat X$.
In what follows we will say that $X$ is the subvariety of $Z$  \emph{defined by} either $f=(f_1,\dots,f_s)$ or $g=(g_1,\dots,g_s)$.

\subsection{Regularity conditions}
We now introduce different conditions on the singularities of a subvariety $X$ of a toric variety $Z$ defined by a set of Laurent polynomials $f=(f_1,\dots,f_s)$.
Given a face $\Gamma=\Gamma_1+\cdots+ \Gamma_s$ of $\Delta(f)$, where  $\Gamma_i$ is a face of  $\Delta_i:=\Delta(f_i)$, the {\em face system of} $f$ {\em associated to} $\Gamma$ is the set $f_{\Gamma}=(f_{1,\Gamma_1},\dots, f_{s,\Gamma_s})$, where 
$f_{i,\Gamma_i}=\sum_{v\in \Gamma_i\cap \zz^n}c_{i,v}T^v$.

\begin{definition}
A subvariety $X$ of a toric variety $Z$ defined by $f=(f_1,\dots,f_s)$ is
\begin{enumerate}
\item 
 {\em non-degenerate} if for any face $\Gamma$ of $\Delta(f)$, the differential of $f_{\Gamma}$ has rank $s$ at any point of $V(f_{\Gamma})\subseteq \mathbb T^n$;
 \item {\em well-formed} if it contains no singular strata of $Z$ of dimension $\dim(X)-1$;
   \item {\em quasismooth} if $\hat X=\hat p^{-1}(X)$ is smooth.   
 \end{enumerate}
\end{definition}

\begin{remark}
This definition of non-degeneracy is the same as $\Delta$-nondegeneracy in 
    \cite{Khov}. By \cite[Lemma 4]{Khov}
    for fixed Newton polytopes $\Delta_i$, the set of collections of coefficients $c_{i,v}$ such that the set $F$ is non-degenerate is open and non-empty. 
\end{remark}

\begin{definition}\label{qsci}
 An irreducible subvariety $X$ of a toric variety $Z$ defined by $g=(g_1,\dots,g_s)$,
where $g_i\in \mathcal R(Z)$, is a {\em quasismooth complete intersection} if:
 \begin{enumerate}
     \item $X$ is well-formed,
\item $X$ is non-degenerate,
     \item the Jacobian matrix of $g$ has maximal rank at every point of $\hat X$.
 \end{enumerate}
\end{definition}

The last condition in the definition implies that $\hat X$ is a smooth complete intersection in $\hat Z$ by the Jacobian criterion. In particular $X$ is quasismooth and $\dim(X)=\dim (Z)-s$. 

\begin{lemma}
If $X\subseteq Z$ satisfies (iii) in Definition \ref{qsci} and ${\rm codim}(\overline X\backslash \hat X)\geq 2$, then $\overline X$ is irreducible and normal, and $X$ is irreducible. Moreover $X=\overline{V(f_1,\dots,f_s)}$. 
\end{lemma}

\begin{proof}
By condition (iii) and the hypothesis on the codimension of $\hat X$ the subvariety $\bar X$ satisfies Serre's criterion for normality (observe that $\bar X$ is Cohen-Macaulay since it is a local complete intersection). Thus $\bar X$ is normal.
This implies that $\bar X$ is irreducible (see the proof of \cite[Theorem 3.9 (i),(iv)]{antican}) and thus $X$ is irreducible. 
Observe that in general $\overline{V(f)}\subseteq X$, but they coincide if $X$ is irreducible.
\end{proof}

\begin{example}
  If $Z=\pp^3$, then one can take 
  \[
  p:\mathbb T^4\to \mathbb T^3,\ (x_1,x_2,x_3,x_4)\mapsto \left(\frac{x_1}{x_4},\frac{x_2}{x_4}, \frac{x_3}{x_4}\right)
  \]
 and the characteristic morphism 
  is the natural quotient map $p:\cc^4\backslash\{0\}\to \pp^3$.
  Let $f_1=T_2^2-T_1T_3, f_2=T_1-T_2T_3, f_3=T_3^2-T_2 \in \cc[T_1^{\pm}, T_2^{\pm}, T_3^{\pm}]$. The homogeneizations of the $f_i$'s in $\mathcal R(Z)$ are   $g_1=x_2^2-x_1x_3, g_2=x_1x_4-x_2x_3, g_3=x_3^2-x_2x_4$. The subvariety of $Z$ defined by $g_1,g_2,g_3$ is the rational normal curve $X$ of degree three. Observe that $X$ is well-formed and quasismooth. An easy computation shows that it is non-degenerate. However $X$ is not a quasismooth complete intersection since property (iii) of Definition \ref{qsci} fails.
  On the other hand the subvariety defined by $f_1,f_2$ is not quasismooth and is the union of $X=\overline{V(f_1,f_2)}$ and the line $x_1=x_2=0$.   
\end{example}

\begin{remark}
There are examples of complete intersections $X$ in weighted projective spaces which are quasismooth  but are not well-formed. 
On the other hand, a complete intersection of dimension at least $3$  in a well-formed weighted projective space which is quasismooth is either well-formed or the intersection of a linear cone with other hypersurfaces (see \cite[Note 6.18 (ii)]{F} or \cite[Theorem 2.17]{Pri}). 
\end{remark}

\subsection{Quasismoothness}\label{2.4}
Since it will be an important property through the paper, we dedicate one section to describe quasismoothness for complete intersections.

Let $Z=Z_{\Sigma}$ be a normal projective toric variety and $\mathcal R(Z)=\cc[x_1,\dots,x_r]$ be its homogeneous coordinate ring, where $\deg(x_i)=\alpha_i\in \Cl(Z)$.
We will give a combinatorial characterization of quasismoothness for
codimension $s\leq r$ complete intersections $X$ in $Z$ 
defined by $g=(g_1,\ldots,g_s)$, where $g_i\in \mathcal R(Z)$ and
 $\deg(g_i)=d_i\in \Cl(Z)$.

Let $\mathcal V(g_i)$ be the vector subspace of $\mathcal R(Z)$ generated by the monomials of $g_i$ and $\mathcal L(g_i)$ be the associated linear system in $\hat Z$, i.e. the set of effective divisors defined as zero loci of polynomials in $\mathcal V(g_i)$. In what follows we will assume the $g_i$'s to be {\em general} elements of $\mathcal V(g_i)$.

\begin{lemma}\label{bertini}
If $g_1,\ldots,g_s\in \mathcal R(Z)$ are general in $\mathcal V(g_1),\dots, \mathcal V(g_s)$, 
then the singular locus of $\hat X$ is contained in the union 
of the base loci of $\mathcal L(g_i), i=1,\ldots,s$.
\end{lemma}

\begin{proof}
It follows from Bertini's Theorem applied to the restriction of the linear system $\mathcal L(g_i)$ to a general 
hypersurface in $\mathcal L(g_j)$ for all $i\neq j$. 
\end{proof}
Let $E_i$ be the divisor defined by $g_i$ in $Z$ and consider the projective bundle 
\[
 Z(E):=\pp(\Osh_Z(E_1)\oplus\cdots\oplus \Osh_Z(E_s))
\] 
over $Z$.
The homogeneous coordinate ring of $Z(E)$ is the polynomial ring in the variables  $y_1,\dots, y_{r+s}= x_1,\dots,  x_r,t_1,\dots,t_s$ 
with degrees given by the columns of the following matrix in $\Cl(Z(E))\cong \Cl(Z)\oplus \zz$:
\[
\left(
\begin{matrix}
\alpha_1 & \alpha_2 & \cdots & \alpha_r & d-d_1 & \cdots & d-d_s\\
0 & 0 & 0 & 0 & 1 & \cdots & 1 
\end{matrix}
\right),
\]
where $d=\sum_id_i$.
Its irrelevant ideal is generated by $\psi(I)$ and $(t_1,\dots, t_s)$, where $I$ is the irrelevant ideal of $Z$ and
$\psi:\mathcal R(Z)\to \mathcal R(Z(E)),\ x_i\mapsto x_i$, see
\cite{mavlyutov}.
The polynomial 
\begin{equation}\label{pot}
F=t_1g_1+\cdots +t_sg_s
\end{equation}
is $\Cl(Z(E))$-homogeneous 
and thus defines a hypersurface $Y$ in $Z(E)$ of degree $(d,1)$.

\begin{proposition}\label{prop-qs}
A well-formed and non-degenerate subvariety $X$ of $Z$ defined by $g_1,\dots,g_s\in \mathcal R(Z)$ is a quasismooth complete intersection if and only if the hypersurface $Y$ defined by $F$ (\ref{pot}) in $Z(E)$ is quasismooth.
\end{proposition}

\begin{proof}
It is a direct computation using the fact that $t_1=\cdots=t_s=0$ is 
empty in $Y$. See also \cite[Lemma 2.2]{mavlyutov}.
\end{proof}

By the definition of $Z(E)$ there is a commutative diagram 
\[
\xymatrix{
\displaystyle\widehat{Z(E)}\ar[r]^{\hat{p}}\ar[d]_{\hat\pi } &  Z(E)\ar[d]^{\pi}&\supseteq Y\\
\hat Z\ar[r]^{\hat p}& Z&\supseteq X,
}
\]
where  $\pi:Z(E)\rightarrow Z$ is the morphism defining the projective bundle,  $\hat \pi$ is the projection on the first $r$ variables and   $\hat p$ denotes the characteristic map for both $Z$ and $Z(E)$.

\begin{remark}
Observe that by Proposition \ref{prop-qs} and \cite[Corollary 3.11]{ACGqs}, quasismoothness of $Y\subset Z(E)$ only depends on the Newton polytope of $F$, and thus of the polynomials $g_i$, $i=1,\ldots, s$.
\end{remark}

Proposition \ref{prop-qs} and  \cite[Theorem 3.6]{ACGqs} provide a combinatorial characterization of quasismoothness for complete intersections. We study in detail the case $s=2$, generalizing Fletcher's results in \cite{F}.
Let $B$ be the base locus of the linear system $\mathcal L(F)$ in $\widehat{Z(E)}$, i.e. the zero set of all monomials in $F$. Similarly, let $B_i$ be the base locus of the linear system generated by the monomials of $g_i$ in $\hat{Z}$, $i=1,2$.
Observe that 
\[
B=(\hat B_1\cap \{t_2=0\})\cup (\hat B_2\cap \{t_1=0\})\cup (\hat B_1\cap \hat B_2),
\]
where $\hat B_i=\hat \pi^{-1}(B_i)$.
In what follows, given a subset $I\subseteq \{1,\dots,r+s\}$, we define
\[
\hat D_I=\{y_\rho=0: \rho \in I\}\subseteq \widehat {Z(E)}
\]
and in case it is not empty, we call it  a {\em toric stratum} in $\widehat{Z(E)}$.
Moreover, we denote by $D_I$ its image in $Z(E)$ by $\hat p$.
We also recall that a set of polytopes in $\mathbb Q^n$ is {\em dependent} if it contains a collection of $l > 0$ non-empty polytopes which can be all translated in an $(l - 1)$-dimensional subspace (see \cite[Definition 3.1]{ACGqs}.

 \begin{proposition}\label{prop1}
The hypersurface $Y\subseteq Z(E)$ is quasismooth  if and only if for any toric stratum  
$\hat D_I \subseteq B$ one of the following holds:
\begin{enumerate}
\item $\hat D_I\subset \hat B_1\backslash \hat B_2$  
and the set of polytopes in \eqref{polytopes-case1} is dependent,

\item $\hat D_I\subset \hat B_2\backslash \hat B_1$
and the set of polytopes in \eqref{polytopes-case2} is dependent,

\item $\hat D_I\subset \hat B_1\cap \hat B_2$ 
and the set of polytopes 
in \eqref{polytopes-case3} is dependent.
\end{enumerate}
\end{proposition}

\begin{proof}
If $\hat D_I\subseteq B$, then one can write $F=\sum_{\rho\in I} y_\rho F^\rho$ and define $\Delta_I(F^\rho)$ as the Newton polytope of the restriction of $F^{\rho}$ to $D_I$.
By \cite[Theorem 3.6]{ACGqs} the hypersurface $Y$ is quasismooth  if and only if the set of non-empty polytopes  
$$\{\Delta_I(F^\rho): \rho\in I\}$$
is dependent
for any $\hat D_I \subseteq B$.
We now distinguish three cases. 

If $\hat D_I\subset (\hat B_1\backslash \hat B_2)\cap\{t_2=0\}$ one can write 
\[
F=t_1g_1+t_2g_2=\sum_{\rho\in J}x_\rho (g_1^\rho t_1)+t_2\left(g'_2+\sum_{\rho\in J} x_\rho g_2^\rho\right),
\]
where $J=I\cap \{1,\dots,r\}$ and $g'_2$ is a non-constant polynomial in the variables $x_\tau$, $\tau\not\in J$.
Thus the set of polytopes considered in \cite[Theorem 3.6]{ACGqs} is
\begin{equation}\label{polytopes-case1}
\{\Delta_{I}(F^\rho): \rho\in I\}=\{\Delta_I(g_1^{\rho}t_1): \rho\in  J\}\cup \{\Delta_I(g'_2)\}
\end{equation}
i.e. the Newton polytope of $g_2'$ and of the restrictions to $\hat D_I$ of $g_1^\rho t_1$.

Analogously, if $\hat D_I\subset (\hat B_2\backslash \hat B_1)\cap\{t_1=0\}$ the polytopes are, up to translation, the Newton polytopes of the restrictions to $\hat D_{I}$ of $g_1$ and of the polynomials $g_2^\rho$:
\begin{equation}\label{polytopes-case2}\{\Delta_{I}(F^\rho): \rho\in I\}=\{\Delta_I(g_2^{\rho}t_2): \rho\in J\}\cup \{\Delta_I(g'_1)\}.
\end{equation}

If $\hat D_I\subset \hat B_1\cap \hat B_2$ one can write 
\[
F=t_1g_1+t_2g_2=\sum_{\rho\in J}x_\rho (g_1^\rho t_1 +g_2^\rho t_2),\quad g_1^\rho,g_2^\rho\in \mathcal R(Z),
\]
thus the collection of polytopes is
\begin{equation}\label{polytopes-case3}
\{\Delta_{I}(F^\rho): \rho\in I\}=
\{\Delta_I(g_1^{\rho}t_1+g_2^\rho t_2): \rho\in J\}
\end{equation}
i.e. the Newton polytopes of the restrictions to $\hat D_I$ of $g_1^\rho t_1+g_2^\rho t_2$.
\end{proof}

\begin{example}
Let $X$ be the subvariety of $\pp(1,1,1,2)$ defined by $g_1=x_1x_2+x_3^2+x_4$ and $g_2=x_1^3+x_3^3+x_2x_4+x_2^2x_3+x_3x_4$. 
We have 
\(
B=\hat{D}_{1}\cup \hat{D}_2\cup \hat{D}_3,
\) where
\[
\hat D_1=\{x_1=x_3=x_4=0\}\subset \hat B_1\cap \hat B_2,
\]
\[
\hat D_2=\{t_1=x_1=x_2=x_3=0\}\subset (\hat B_2\setminus \hat B_1)\cap \{t_1=0\},
\]
\[
\hat D_3=\{t_2=x_2=x_3=x_4=0\}\subset (\hat B_1\setminus \hat B_2)\cap \{t_2=0\}.
\]
First consider \(\hat D_1\), which corresponds to Case (iii) of Proposition \ref{prop1}. One can write
\[
F
= x_1(t_1x_2+t_2x_1^{2})
+ x_3(t_1x_3+t_2x_4+t_2x_2^{2}+t_2x_3^{2})
+ x_4(t_1+t_2x_2),
\]
and when restricting to \(\hat D_1\) we are left with the following three polytopes in \(\mathbb{R}^3\)
\[
\Delta(t_1x_2),\ \Delta(t_2x_2^{2}),\ \Delta(t_2x_2),
\]
 each reduced to a single point. They can be translated in a
\(0\)-dimensional space, thus they form a dependent collection.
Now consider \(\hat D_2\), which corresponds to Case (ii) of Proposition \ref{prop1}. One can write
\[
F
= t_1(x_1x_2+x_3^{2}+x_4)
+ x_1(t_2x_1^{2})
+ x_2(t_2x_2x_3+t_2x_4)
+ x_3(t_2x_4+t_2x_3^{2}),
\]
and when restricting to \(\hat D_2\) we are left with
\[
\Delta(x_4),\ \varnothing,\ \Delta(t_2x_4),\ \Delta(t_2x_4).
\]
Here we have 3 polytopes in \(\mathbb{R}^2\), all reduced to a point, hence again a dependent collection. Finally consider \(\hat D_3\) which corresponds to Case (i) of Proposition \ref{prop1}. One can write
\[
F
= x_2(t_1x_1)+x_3(t_1x_3)+x_4(t_1)
+ t_2(x_3x_4+x_1^{3}+x_2^{2}x_3+x_3^{3}+x_2x_4),
\]
and when restricting to \(\hat D_3\) we are left with
\[
\Delta(t_1x_1),\ \varnothing,\ \Delta(t_1),\ \Delta(x_1^{3}),
\]
that is again 3 points in \(\mathbb{R}^2\) forming a dependent collection.
Since all the toric strata satisfy the conditions of  Proposition \ref{prop1}, we can conclude that $X$ is quasismooth.
\end{example}

Given a toric stratum $\hat D_J\subseteq B_1$, we will denote by $k_J(g_1)$ 
the number of $\rho\in J$ such that $\Delta_{J}(g_1^{\rho})$ is not empty 
(that is the restriction of $g_1^{\rho}$ to $\hat D_{J}$ is not zero).
Similarly we define $k_{J}(g_2)$ when $\hat D_{J}\subseteq \hat B_2$.
Finally, when $\hat D_J\subseteq B_1\cap B_2$ we define $k_{J}(g_1,g_2)$ as the number of 
$\rho\in J$ such that either $\Delta_{J}(g_1^{\rho})$ or $\Delta_{J}(g_2^{\rho})$ is not empty.
Observe that $k_{J}(g_1,g_2)\geq  k_{J}(g_i),i=1,2$.

\begin{remark}\label{rem-dim}
  Let $\hat D_I$ be a toric stratum contained in $B\subseteq \widehat{Z(E)}$. Then 
the dimension of its image $D_{I}$
in $Z(E)$ is:

$$\dim (D_I)=
\begin{cases}
    \dim (\pi(D_I))& \mbox{if } \hat D_I\subseteq \{t_i=0\}, i=1 \mbox{ or } 2,\\
    \dim (\pi(D_I))+1& \mbox{otherwise.}
\end{cases}$$
\end{remark}

\begin{theorem}\label{cor-qs}
A well-formed subvariety $X\subseteq Z$ defined by general $g_1,g_2\in \mathcal R(Z)$ is a quasismooth complete intersection if for any toric stratum $\hat D_J\subseteq B_1\cup B_2$, with $\dim  (D_J)=d_{J}$, one of the following holds 
\begin{enumerate}
\item $\hat D_J\subseteq B_1\backslash B_2$  and $k_{J}(g_1)\geq d_J$,
\item $\hat D_I\subseteq B_2\backslash B_1 $ and $k_{J}(g_2)\geq d_J$,
\item $\hat D_J\subseteq B_1\cap B_2$, $k_{J}(g_i) \geq d_{J}+1$ for $i=1,2$
and $k_{J}(g_1,g_2)\geq d_{J}+2$.
\end{enumerate}
If $Z$ is a fake weighted projective space, then $\dim(D_J)=\dim(\hat D_J)-1$ 
and the previous condition is also necessary.
\end{theorem}

\begin{proof}
Given any non-zero homogeneous element $h\in \mathcal R(Z(E))$, 
the polytope $\Delta_{I}(h)$  
of its restriction to $\hat D_{I}$
is contained in an affine space of dimension $\dim (D_{I})$.
Moreover, it is clear that a set of at least $r+1$ non-empty 
polytopes in an $r$-dimensional space is dependent.

Observe that $\hat D_I\subseteq \hat B_1\setminus \hat B_2$ if and only if $\hat D_J\subseteq B_1\backslash B_2$,
where $J=I\cap\{1,\dots,r\}$. The number of non-empty 
polytopes in \eqref{polytopes-case1} 
is equal to $k_{I}(g_1)+1=k_J(g_1)+1$.
Thus the set of polytopes is dependent if $k_{J}(g_1)+1\geq d_J+1$.
Similarly for case (ii).

In case (iii), 
if $\hat D_I\subseteq (\hat B_1\cap \hat B_2)\cap \{t_1=0\}$, then the number of non-empty polytopes in \eqref{polytopes-case3}  is $k_J(g_2)$ and they live in a space of dimension $d_J$.
Similarly if $\hat D_I\subseteq (\hat B_1\cap \hat B_2)\cap \{t_2=0\}$.
On the other hand, if  $\hat D_I\subseteq \hat B_1\cap \hat B_2$ but neither $t_1$ nor $t_2$ vanishes along $\hat D_I$, then $k_{J}(g_1,g_2)$ 
counts the number of non-empty 
polytopes in \eqref{polytopes-case3}
and the polytopes in this set live in a space of dimension $d_J+1$ according to Remark \ref{rem-dim}.

We now prove the last statement.
Assume that either (i), (ii)  or (iii) fail for some $\hat D_{J}\subseteq B_1\cup B_2$.
For example assume that $\hat D_J\subseteq B_1\setminus B_2$  and $k_{J}(g_1)< \dim(\hat D_J)-1$. 
Let $I=J\cup\{r+1\}$, so that $\hat D_I\subseteq B$.
Observe that the intersection of the singular locus of $F$ with $\hat D_I$ is:
\[
\hat D_I\cap \{g_2=0\}\cap \{g_1^\rho=0: \rho \in J\}.
\]
If $S$ is an irreducible component of such locus, then
\[
\dim(S)\geq \dim(\hat D_{I})-k_{J}(g_1)-1=\dim(\hat D_{J})-k_{J}(g_1)-1>0,
\]
thus $X$ is not quasismooth.
Similarly when $\hat D_{J}\subseteq  B_2\setminus  B_1$.
Now assume that $\hat D_{J}\subseteq  B_1\cap  B_2$.
If $k_{J}(g_i)<\dim(\hat D_{J})$ for some $i$, 
then by \cite[Corollary 5.3]{ACGqs} the hypersurface defined by $g_i$ is not quasismooth along $\hat D_J$ and thus $X$ is not quasismooth.
If  $k_{J}(g_i)\geq \dim(\hat D_{J})$ 
for $i=1,2$ and $k_{J}(g_1,g_2)\leq  \dim (\hat D_{J})$, 
the only possibility is 
\[k_{J}(g_1)=k_{J}(g_2)=k_{J}(g_1,g_2)= \dim \hat D_{J}=:k\]
i.e. the indices $\rho\in J$ such that $\Delta_J(g_i^{\rho})\neq\emptyset$ are the same for $i=1,2$.
Thus if $S$ is a component of the intersection of the singular locus of $F$ with ${\hat\pi}^{-1}(\hat D_J)\subseteq \hat B_1\cap \hat B_2$, then
\[
\dim(S)\geq \dim(\hat \pi^{-1}(\hat D_{J}))-k=\dim(\hat D_{J})+1-k>0,
\]  
and $X$ is not quasismooth.
\end{proof}

Translating the condition in Theorem \ref{cor-qs} in terms of monomials, one obtains the following result (see also \cite[Theorem 8.7]{F}).

\begin{corollary}
Let $Z$ be a fake weighted projective space with homogeneous cordinates $x_0,\dots, x_n$, where $n\geq 2$, and $X\subseteq Z$ be a well-formed subvariety defined by general $g_1,g_2\in \mathcal R(Z)$. Then $X$ is a quasismooth complete intersection if and only if for any non-empty subset $I$ of $\{0,\dots,n\}$ one of the following holds
\begin{enumerate}
\item $g_1$ and $g_2$ both contain a monomial of the form $x_I^{M}$, i.e. in the variables $x_i$, $i\in I$;
\item $g_2$ contains a monomial of the form $x_I^{M}$ and $g_1$ contains monomials of the form $x_I^{M_1}x_{i_1},\dots,  x_I^{M_d}x_{i_d}$, where $i_1,\dots,i_d$ are distinct; 
\item  $g_1$ contains a monomial of the form $x_I^{M}$ and $g_2$ contains  (at least)  monomials  of the form $x_I^{M_1}x_{i_1},\dots,  x_I^{M_d}x_{i_d}$, where $i_1,\dots,i_d$ are distinct;
\item $g_1$ contains $d+1$ monomials of the form $x_I^{M_1}x_{i_1},\dots,  x_I^{M_{d+1}}x_{i_{d+1}}$ and
$g_2$ contains $d+1$ monomials of the form $x_I^{N_1}x_{j_1},\dots,  x_I^{N_{d+1}}x_{j_{d+1}}$ such that  $i_1,\dots,i_{d+1}$ are distinct, $j_1,\dots,j_{d+1}$ are distinct and 
$\{i_1,\dots,i_{d+1},j_1,\dots, j_{d+1}\}$ has cardinality at least $d+2$;
\end{enumerate}
where $d:=|I|-1$. 
\end{corollary}

\section{Generalized nef partitions}\label{sec-gnp}
In this section we generalize the notion of nef partition given in 
\cite{BB, Bor1} to the case when the polytope is not reflexive but its polar polytope is canonical. 
Equivalently, we assume the 
toric variety to be $\qq$-Fano with canonical singularities and not necessarily Fano.
We first deal with polytopes, while the geometric interpretation will be given in Section \ref{sec-geometry}.

\subsection{Definition and first properties}
Let $\Delta\subset M_\qq$ be a full-dimensional lattice polytope, i.e. whose vertices are points of $M$, which  contains the origin as an interior point. The {\em polar} polytope of $\Delta$ is \[\Delta^\circ:=\{n\in N_\qq: \langle m,n\rangle\geq -1\ \forall m\in \Delta\}.\]
  The polytope $\Delta$ is called {\em $\qq$-Fano} if its vertices are primitive in $M$, {\em canonical} if the origin is its unique interior lattice point and {\em reflexive} if $\Delta^\circ$ is a lattice polytope.
In what follows we will denote by $\ell^*(\Delta)$ the number of lattice points in the relative interior of $\Delta$.

Let $\Delta\subseteq M_{\qq}$ be a polytope containing the origin as an interior point. 
Consider a partition of the set of vertices of $\Delta^\circ$:
 \[
 {\rm Vert}(\Delta^\circ)=I_1\sqcup\ldots\sqcup I_s.
 \]
For $i=1,\ldots,s$ let 
\begin{equation}\label{deltai}
\Delta_i:=\{m\in M_{\rr}:  \langle  m,n_\rho\rangle\geq -{\bf 1}_{i}(n_\rho) \mbox{ for all }n_\rho\in {\rm Vert}(\Delta^\circ)\}
\end{equation}
with ${\bf 1}_{i}(n_\rho)=1$ if $n_\rho\in I_i$ and 0 otherwise.

\begin{proposition}\label{car}
The following properties hold:
\begin{enumerate}
\item $0\in \Delta_i$ for any $i$;
\item $\Delta_i\cap\Delta_j=\{0\}$ if $i\not=j$;
\item $\Delta_1+\cdots+\Delta_s\subseteq \Delta$;
\item $\Delta_1+\cdots+\Delta_s=\Delta$ if and only if 
for any $i\in\{1,\ldots,s\}$ and for any facet $\sigma$ of $\Delta^\circ$ 
there exists a vertex $m_{i,\sigma}$ of $\Delta_i$ such that 
\begin{equation}\label{gnp}\langle m_{i,\sigma},n_{\rho}\rangle =-{\bf 1}_{i}(n_\rho),\ \forall n_\rho\in{\rm Vert}(\sigma).
\end{equation}
\end{enumerate}
\end{proposition}

\begin{proof}
Property (i) is clear from the definition.
If $m\in \Delta_i\cap\Delta_j$ with $i\not=j$, then  $\langle m, n_\rho \rangle\geq 0$ for all $n_\rho\in {\rm Vert}(\Delta^\circ)$ and thus $m=0$, proving (ii). 

Given $m_i\in \Delta_i$, $i=1,\dots,s$, we have that $\langle \sum_{i=1}^s m_i, n_{\rho}\rangle \geq -1$ for all $n_\rho\in {\rm Vert}(\Delta^\circ)$. Thus $\sum_{i=1}^s m_i\in (\Delta^\circ)^\circ=\Delta$, proving (iii).

To prove (iv), first assume that $\Delta_1+\cdots+\Delta_s=\Delta$
and let $\sigma$ be a facet of $\Delta^{\circ}$ and $m_{\sigma}\in \Delta$ be the corresponding polar vertex. We have that $m_\sigma=m_{1,\sigma}+\cdots+m_{s,\sigma}$, where $m_{i,\sigma}\in \Delta_i$. If $n\in I_i\cap \sigma$, then
\[
-1=\langle m_{\sigma},n\rangle=\sum_{j=1}^s\langle m_{j,\sigma},n\rangle\geq \langle m_{i,\sigma},n\rangle\geq -1,
\]
by definition of $\Delta_j$.
Thus all inequalities must be equalities and the $m_{i,\sigma}$'s satisfy the condition (\ref{gnp}).
Conversely, assume that condition (\ref{gnp}) holds.
As observed before we have that $\Delta_1+\cdots+ \Delta_s\subseteq \Delta$.
On the other hand, given a vertex $v\in \Delta$, let $\sigma_v$ be the corresponding facet of $\Delta^\circ$ and let $m_{i,\sigma_v}$ be the vertex of $\Delta_i$ as in (\ref{gnp}).
Then $\langle\sum_i m_{i,\sigma_v},n_{\rho}\rangle=-1$ for all $n_{\rho}\in {\rm Vert}(\sigma_v)$, thus $\sum_i m_{i,\sigma_v}=v$.
This proves that $\Delta\subseteq \Delta_1+\cdots+\Delta_s$.
\end{proof}
 
\begin{definition}\label{gnp2}
The set of polytopes $\Pi(\Delta)=\{\Delta_1,\ldots,\Delta_s\}$ 
associated to a partition of  ${\rm Vert}(\Delta^\circ)$ as in (\ref{deltai})
is a 
{\em generalized nef partition} of $\Delta$ if $\Delta=\Delta_1+\cdots+\Delta_s$ (or if the equivalent condition in Proposition \ref{car}\,(iv) holds).
\end{definition}

\begin{remark}\label{rmk-vertices}
As observed in the proof of the Proposition \ref{car}, if $\sigma$ is a facet of $\Delta^\circ$ and $m_{i,\sigma}$ is as in Proposition \ref{car} (iv), then $m:=\sum_{i=1}^s m_{i,\sigma}$ is the unique vertex of $\Delta$ such that $\langle m,n_{\rho}\rangle=-1$ for any $n_{\rho}\in {\rm Vert}(\sigma)$. 
On the other hand, given any vertex $m_i$ of $\Delta_i$, by a standard property of Minkowski sums, there exists $m'\in \sum_{j\not=i}\Delta_j$ such that $m_i+m'$ is a vertex of $\Delta$. Let $\sigma$ be the polar facet of $\Delta^\circ$, then
\[
m_i+m'=\sum_{j=1}^s m_{j,\sigma}.
\]
By the uniqueness of the decomposition we deduce that $m_i=m_{i,\sigma}$. 
Thus, any vertex of $\Delta_i$ is of the form $m_{i,\sigma}$ for some facet $\sigma$ of $\Delta^\circ$.
 \end{remark}

\begin{remark}\label{rem-gnp}
The authors in \cite[Proposition 3.2]{KRS} prove that if $\Delta$ is reflexive, then Definition \ref{gnp2} is equivalent to give $P_1,\dots,P_s$ lattice polytopes such that $\Delta=P_1+\ldots+P_s$ and $P_i\cap P_j=\{0\}$ for $i\not=j$.

This is false in general if $\Delta$ is canonical but not reflexive.
Let 
\[\begin{array}{l}
P_1=\Conv((0,0,0),(1,1,0),(1,0,1),(0,1,0),(1,2,1))\\
P_2=\Conv((0,0,0),(0,-1,0),(0,0,-1),(-1,-1,0)).
\end{array}
\]
One has $P_1\cap P_2=\{0\}$ and 
$\Delta=P_1+P_2$ is canonical since its vertices are lattice points and the only lattice point in its interior is the origin, but $\Delta$ is not reflexive. One can see that  $n=\left(0,-\frac12,\frac12\right)$ is a vertex of $\Delta^\circ$ such that $\langle m,n\rangle\neq -1$ for any $m\in{\rm Vert}(P_i),\ i=1,2$. Thus there is no partition of the vertices of $\Delta^\circ$ whose associated polytopes $\Delta_i$, as in \eqref{deltai}, are $P_1,P_2$.

\end{remark}

 \begin{example}\label{exP113}
Let $Z=\pp(1,1,3)$ and $\Delta:=\Conv((-1,-3), (-1,2),(2/3, 1/3))$ be its anticanonical polytope. Observe that $\Delta$ is not reflexive.
The vertices of $\Delta^\circ$ are the minimal generators of the fan of $Z$:
\[
n_1=(1,0),\ n_2=(-1,-1),\ n_3=(-2,1).
\]
Consider the partition of $\text{Vert}(\Delta^\circ)$ given by
 $I_1=\{n_1\}, I_2=\{n_2,n_3\}$. By \eqref{deltai} one obtains the polytopes (in light gray in Figure \ref{fig113})
\[
\Delta_1=\Conv((0,0),(-1,-2),(-1,1)),\  \Delta_2=\Conv((0,-1), (0,1), (2/3,1/3)),
\]
which define a generalized nef partition of $\Delta$ since $\Delta_1+\Delta_2=\Delta$. 
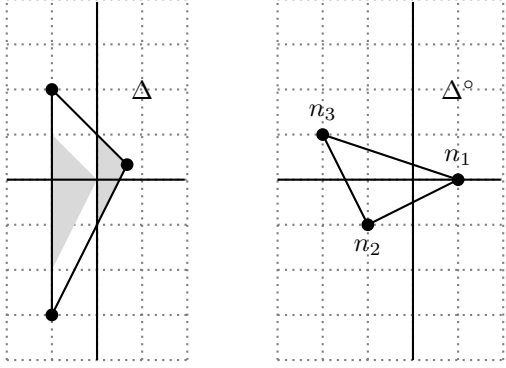
\begin{figure}[ht]
\centering\begin{tikzpicture}[>=stealth',shorten >=1pt,auto,node distance=1.5cm, thick,main node/.style={circle,draw,font=\bfseries},xscale=.6,yscale=.6]

\draw[step=1cm,gray,dotted] (-2,-4) grid (2,4);

\node at (1,2) {$\Delta$};

\fill[gray!30] (0,0)-- (-1,-2)--(-1,1);
\fill[gray!30] (0,-1)-- (0,1)--(2/3, 1/3);

\draw (-1,-3)--(-1,2);
\draw (-1,-3)--(2/3, 1/3);
\draw (-1,2)--(2/3, 1/3);
\node at (-1,-3) {\pgfuseplotmark{*}}; 
\node at (-1,2) {\pgfuseplotmark{*}};
\node at (2/3, 1/3) {\pgfuseplotmark{*}};

\draw (0,-4)--(0,4);
\draw (-2,0)--(2,0);

\draw [dotted] (0,-2)--(0,3);


\draw[step=1cm,gray,dotted] (4,-4) grid (9,4);
\node  at (8,0) {\pgfuseplotmark{*}};
\node at (8,0.5) {$n_1$};

\node at (6,-1) {\pgfuseplotmark{*}};
\node at (6,-1.5) {$n_2$};
\node at (5,1) {\pgfuseplotmark{*}};
\node at (5,1.5) {$n_3$};

\draw [thick](8,0)--(6,-1);
\draw [thick] (8,0)--(5,1);
\draw [thick](5,1)--(6,-1);

\draw (7,-4)--(7,4);
\draw (4,0)--(9,0);


\node at (8,2) {$\Delta^\circ$};

\end{tikzpicture}
\caption{Generalized nef partition $\Pi(\Delta)$}
\label{fig113}
\end{figure}
Observe that one can also consider different number of subsets of ${\rm Vert}(\Delta^\circ)$ for example 
\begin{equation}\label{eqs=3}
I_1\sqcup I_2\sqcup I_3=\{n_1\}\sqcup\{n_2\}\sqcup\{n_3\}.\end{equation}
With \eqref{deltai} one gets
\[
\Delta_1=\Conv((0,0), (0,1), (1/3, 2/3)),\ 
\Delta_2=\Conv((0,0), (0,-1), (1/3, -1/3)),\] 
\[\Delta_3=\Conv((0,0), (-1,-2), (-1,1))).\ 
\]
This is a generalized nef partition of $\Delta$ sicne $\Delta=\sum_{i=1}^3\Delta_i$.

This can be observed more in general: for $Z=\pp(1,1,n)$, with $n\geq3$, the anticanonical polytope $\Delta$ of $Z$ is not reflexive and each partition of the vertices of $\Delta^\circ$ gives a generalized nef partition.

\end{example}

\begin{example}\label{ex-part}
Consider the reflexive polytope
 $$\Delta = \text{Conv}((2, 1), (-1, 1), (-1, -1), (0, -1)),$$ 
 whose normal fan defines the toric variety $Z = \text{Bl}_p(\mathbb{P}^2)$, the blow-up of $\mathbb P^2$ at one point $p$.
The vertices of $\Delta^\circ$ are the minimal generators of the rays of the fan of $Z$:
 \[
n_1=(0,-1),\ n_2=(-1,1),\ n_3=(1,0),\ n_4=(0,1).
 \]
Consider the partition of $\text{Vert}(\Delta^\circ)$ given by $I_1 = \{n_2,n_4\}, I_2 = \{n_1, n_3\}$. 
 By \eqref{deltai} one obtains the following polytopes: 
$$\Delta_{1} = \text{Conv}((0,0), (0,-1),(1,0)),\ \Delta_{2} = \text{Conv}((0,0), (-1,1), (-1,0), (1,1)).$$

Observe that $\Pi(\Delta)=\{\Delta_1,\Delta_2\}$ is a generalized nef partition (actually a nef partition in the classical sense of \cite{BB}) since $\Delta=\Delta_1+\Delta_2$.

On the other hand, if we consider the partition of ${\rm Vert}(\Delta^\circ)$ given by $I'_1=\{n_2,n_3\}$ and $I'_2=\{n_1,n_4\}$, then $\Delta'_1+\Delta'_2\neq\Delta$ thus $\Delta'_1,\Delta'_2$ does not give a generalized nef partition of $\Delta$. Equivalently, if $i=2$ and $\sigma$ is the facet of $\Delta^\circ$ with vertices $n_3$ and $n_4$, there is no  vertex $m_{i,\sigma}$ of $\Delta'_2$ as in Proposition \ref{car} (iv).

Up to symmetries, there are two more possible generalized nef partitions of $\Delta$, associated to the partitions
$\{n_1\}\sqcup\{n_2,n_3,n_4\}$
and $\{n_2\}\sqcup\{n_1,n_3,n_4\}$.

\end{example}

\subsection{Duality}\label{duality}
Let $\Pi(\Delta)=\{\Delta_1,\ldots,\Delta_s\}$ 
be a generalized nef partition of $\Delta$ associated to  
a partition $I_1\sqcup\ldots\sqcup I_s={\rm Vert}(\Delta^\circ)$. We now introduce a set of polytopes $\nabla_1,\dots,\nabla_s$ dual to  
$\Delta_1,\dots,\Delta_s$.

\begin{definition}\label{def-nabla}
For $j=1,\ldots,s$ define 
\[
\nabla_j:=\{n\in N_\rr: \langle m,n\rangle\geq -\delta_{i,j} \ \forall m\in\Delta_i,\ i=1,\ldots,s\}
\]
and let $\nabla:=\nabla_1+\ldots+\nabla_s$.
\end{definition}
\noindent Observe that $0\in \nabla_j$ and that $\nabla_j\subseteq \Delta^\circ$, for any $j$.

\begin{proposition}\label{prop:origin}
The following results hold:
\begin{enumerate} 
\item $\nabla_j={\rm Conv}(0,I_j)$ for all $j$.
\item $\Delta^\circ={\rm Conv}(\nabla_1,\dots,\nabla_s)$.
\item $\nabla$ and $\nabla^\circ$ contain the origin in their interior.
\end{enumerate}
\end{proposition}

\begin{proof}
    Fix $j\in\{1,\ldots,s\}$. 
    Observe that for all $n_\rho\in I_j$,  $i\in\{1,\ldots,s\}$ and $ m\in\Delta_i$ one has 
    $\langle m,n_\rho\rangle\geq -${\bf 1}$_{I_j}(n_\rho)=-\delta_{i,j}$. Since moreover $0\in\nabla_j$, we have proven $\Conv(0,I_j)\subseteq \nabla_j$.

    Now take $n\in N_\rr\backslash \Conv(0,I_j)$. We prove that $n\notin\nabla_j$. Since $\Cone(\Delta^\circ)=N_\rr$, then $n=\sum_{n_\rho\in\alpha}\lambda_\rho n_\rho$, where $\alpha$ is a face of $\Delta^\circ$ and $\lambda_\rho$ are positive real numbers. 
    If $n\not\in\Delta^\circ$, then
    $n \notin \nabla_j$ by the previous remark.
Let us now assume that $n \in \Delta^\circ$. Since $n \notin \text{Conv}(0,I_j)$, we have $\alpha \nsubseteq \text{Conv}(I_j)$, i.e., there exists $n_{\rho_0} \in \alpha$ and $i \neq j$ such that $n_{\rho_0} \in I_i$. Then for any facet $\sigma$ of $\Delta^\circ$ containing $\alpha$, we have
\[
\langle m_{i, \sigma}, n \rangle = \sum_{n_\rho \in \alpha} \lambda_\rho \langle m_{i, \sigma}, n_\rho \rangle \leq \lambda_{\rho_0} \langle m_{i, \sigma}, n_{\rho_0} \rangle,
\]
because for all $n_\rho \in \alpha$, $\lambda_\rho > 0$ and $\langle m_{i, \sigma}, n_\rho \rangle \in \{0,-1\}$. Finally, since $n_{\rho_0} \in \sigma \cap I_i$, we have $\langle m_{i, \sigma}, n_{\rho_0} \rangle = -1$, so we get
\[
\langle m_{i, \sigma}, n \rangle \leq -\lambda_{\rho_0} < 0,
\]
which shows that $n$ is not in $\nabla_j$ in this case either and completes the proof (i).
Item (ii) is an immediate consequence of  (i).
Finally, observe that $\Delta^\circ\subseteq \nabla$ contains the origin in its interior, thus the same is true for $\nabla$. By duality this holds also for $\nabla^\circ$, proving (iii).
\end{proof}

\begin{proposition}\label{dualgnp}
$\Pi(\nabla)=\{\nabla_1,\ldots,\nabla_s\}$ 
is a generalized nef partition of $\nabla$ and for $i=1,\ldots,s$, we have $$\Delta_i = \{ m \in M_\rr \, | \, \langle m, n \rangle \geq -\delta_{i,j} \text{ for all } n \in \nabla_j, \, j = 1,\ldots,s \}.$$
\end{proposition}

The proof of Proposition \ref{dualgnp} is done in two steps, each stated as a lemma.

\begin{lemma}\label{lemma-part}
$\nabla=\bigcap_{i=1}^{s}\Delta_i^\circ$ and $\nabla^\circ={\rm Conv}(\Delta_1,\ldots,\Delta_s)$.
\end{lemma}

\begin{proof}
Let $n=n_1+\cdots+n_s\in \nabla$ with $n_j\in \nabla_j$ for $j=1,\dots,s$.
Given $m\in \Delta_i$ we have 
\[
\langle m,n\rangle=\sum_{j=1}^s\langle m,n_j\rangle\geq -1
\]
by definition of $\nabla_j$. 
Thus $n\in \Delta_i^\circ$ for all $i$.

Viceversa, let $n\in \bigcap_{i=1}^{s}\Delta_i^\circ\subseteq N_{\rr}$. Since ${\rm Cone}(\Delta^\circ)=N_{\rr}$, $n$ belongs to the cone over a facet $\sigma$ of $\Delta^\circ$. Thus we can write $n=\sum_{n_\rho\in \sigma}\lambda_{\rho} n_{\rho}$, where $\lambda_{\rho}$ are non-negative integers. 
By the definition of generalized nef partition there, for any $i$ there exists $m_{i,\sigma}\in {\rm Vert}(\Delta_i)$ such that $\langle m_{i,\sigma}, n_{\rho}\rangle=-{\bf 1}_i(n_{\rho})$ for any $n_{\rho}\in \sigma$. Thus
\[
\langle m_{i,\sigma},n\rangle=-\sum_{\rho\in \sigma\cap I_i}\lambda_\rho\geq -1.
\]
Let $n_j:=\sum_{n_\rho\in \sigma\cap I_j}\lambda_{\rho}n_{\rho}$. if $m\in \Delta_i$, then
\[
\langle m,n_j\rangle=\sum _{n_\rho\in \sigma\cap I_j}\lambda_{\rho}\langle m,n_{\rho}\rangle
\geq  -\sum _{n_\rho\in \sigma\cap I_j}\lambda_{\rho}{\bf 1}_i(n_{\rho})\geq -\delta_{ij}.
\]
Thus $n_j\in \nabla_j$ and $n=\sum_{j=1}^s n_j\in \nabla$.

The second equality in the statement follows from the first one observing that 
the polar of a convex hull of polytopes is the intersection of their polar polytopes and that polarity is an involution.
\end{proof}

\begin{lemma}\label{4.12}
    There is a partition $J_1\sqcup\ldots\sqcup J_s$ of the vertices of $\nabla^\circ$ such that $J_i={\rm Vert}(\Delta_i)\backslash \{0\}$ for all $i=1,\ldots,s$. Moreover for all $j=1,\ldots,s$:
\[\nabla_j=\{n\in N_{\rr}: \langle m,n\rangle\geq -{\bf 1}_{J_j}(m) \mbox{ for all }m\in{\rm Vert}(\nabla^\circ)\}.
\]
\end{lemma}

\begin{proof}
By Lemma \ref{lemma-part}, we clearly have a partition $\text{Vert}(\nabla^\circ) = \sqcup_i J_i$, where $J_i: = \text{Vert}(\nabla^\circ) \cap \Delta_i \subseteq \text{Vert}(\Delta_i)\backslash\{0\}$,
so we just have to prove that for all $i \in \{1, \ldots, s\}$, any vertex $m_{i}$ of $\Delta_i$ distinct from the origin is in $J_i$. 
Since such an $m_{i}$ is in $\nabla^\circ$, we can write $m_{i} = \sum_{v \in \text{Vert}(\nabla^\circ)} \lambda_v v$, with $\lambda_v \geq 0$ for all $v$ and with $\sum_{v \in \text{Vert}(\nabla^\circ)} \lambda_v = 1$. This can again be written $m_{i} = \sum_{j=1}^s \sum_{v \in J_j} \lambda_v v$. 

Now observe that by Remark \ref{rmk-vertices}, $m_i=m_{i,\sigma}$ for some facet $\sigma$ of $\Delta^\circ$. By definition $m_{i,\sigma}$ is equal to the origin if and only if the intersection $\sigma \cap I_i$ is empty. In particular, since $m_{i,\sigma} \neq 0$, there exists a vertex $n$ of $\Delta^\circ$ contained in $\sigma \cap I_i$ and we have $\langle m_{i,\sigma}, n \rangle = -1$. We get
\[
-1 = \sum_{j=1}^s \sum_{v \in J_j} \lambda_v \langle v, n \rangle \geq \sum_{j=1}^s \sum_{v \in J_j} \lambda_v (-\delta_{i,j}) = - \sum_{v \in J_i} \lambda_v \geq -1.
\]
It follows that all the inequalities above are equalities. In particular, $\sum_{v \in J_i} \lambda_v = 1$, which means that $m_{i,\sigma} = \sum_{v \in J_i} \lambda_v v$ is contained in the convex hull of $J_i$. Since it is a vertex of $\Delta_i$, it must be an element of $J_i$, which concludes the proof of the first statement.
The second statement is an immediate consequence of the first one since every vertex $m$ of $\nabla^\circ$ belongs to some $J_i$, hence verifies $-\mathbf{1}_{J_j}(m) = -\delta_{i,j}$.
\end{proof}

\begin{definition}
We will say that $\Pi(\nabla)=\{\nabla_1,\ldots,\nabla_s\}$ is the {\em dual} of the generalized nef partition $\Pi(\Delta)=\{\Delta_1,\ldots,\Delta_s\}$.
\end{definition}

 As a direct consequence of Definition \ref{def-nabla} and Proposition \ref{dualgnp}, the dual of $\Pi(\nabla)$ is again $\Pi(\Delta)$,
 thus we defined an involution in the set of generalized nef partitions.

\begin{example}[{\bf Continuation of Example \ref{exP113}}]\label{ex113-nabla}
Let $\Pi(\Delta)=\{\Delta_1,\Delta_2\}$ be the nef partition of Example \ref{exP113} associated to $I_1\sqcup I_2$. By Definition \ref{def-nabla} one gets
$$\nabla_{1} = \text{Conv}((0,0), (1,0))\mbox{ and } \nabla_{2} = \text{Conv}((0,0), (-1,-1), (-2,1)),$$ 
and 
$
\nabla = \nabla_{1} + \nabla_{2} = \text{Conv}((0,1),(-1, -2), (0, -1), (-1, 1), (2/3, 1/3))$
(see Figure \ref{fig113dual}).
The dual partition $J_1\sqcup J_2$ of vertices of $\nabla^\circ$ is 
\[
J_1\sqcup J_2=\{(0,-1),(0,1), (2/3, 1/3)\}\sqcup\{(-1,1),(-1,-2))\}.
\]

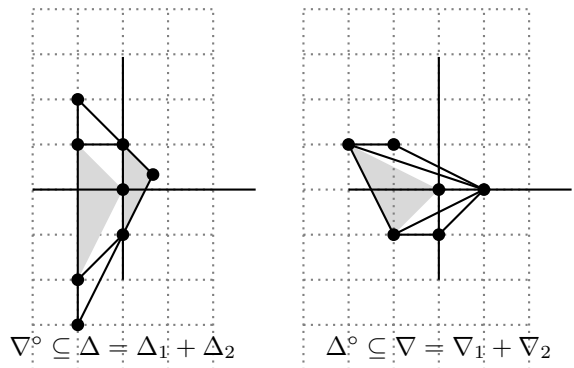
\begin{figure}[ht]
\centering\begin{tikzpicture}[>=stealth',shorten >=1pt,auto,node distance=1.5cm, thick,main node/.style={circle,draw,font=\bfseries},xscale=.6,yscale=.6]

\draw[step=1cm,gray,dotted] (-2,-4) grid (2,4);

\node at (0,-3.5) {$\nabla^\circ\subseteq\Delta=\Delta_1+\Delta_2$};

\fill[gray!30] (0,0)-- (-1,-2)--(-1,1);
\fill[gray!30] (0,-1)-- (0,1)--(2/3, 1/3);

\draw (-1,-3)--(-1,2);
\draw (-1,-3)--(2/3, 1/3);
\draw (-1,2)--(2/3, 1/3);
\node at (-1,-3) {\pgfuseplotmark{*}}; 
\node at (-1,2) {\pgfuseplotmark{*}};
\node at (2/3, 1/3) {\pgfuseplotmark{*}};

\node at (0, 0) {\pgfuseplotmark{*}};

\draw (0,-2)--(0,3);
\draw (-2,0)--(3,0);

\draw (-1,1)--(0,1);
\draw (-1,-2)--(0,-1);
\node at (-1,1) {\pgfuseplotmark{*}}; 
\node at (0,1) {\pgfuseplotmark{*}}; 
\node at (-1,-2) {\pgfuseplotmark{*}}; 
\node at (0,-1) {\pgfuseplotmark{*}}; 

\fill[gray!30] (6,-1)-- (5,1)--(7,0);

\draw[step=1cm,gray,dotted] (4,-4) grid (9,4);
\node  at (8,0) {\pgfuseplotmark{*}};
\node at (6,-1) {\pgfuseplotmark{*}};
\node at (5,1) {\pgfuseplotmark{*}};

\node at (7,0) {\pgfuseplotmark{*}};

\draw [thick](8,0)--(6,-1);
\draw [thick] (8,0)--(5,1);
\draw [thick](5,1)--(6,-1);

\draw (7,-2)--(7,3);
\draw (5,0)--(10,0);

\draw (5,1)--(6,1);
\draw (8,0)--(6,1);
\draw (6,-1)--(7,-1);
\draw (8,0)--(7,-1);

\node at (6,1) {\pgfuseplotmark{*}}; 
\node at (7,-1) {\pgfuseplotmark{*}};

\node at (7,-3.5) {$\Delta^\circ\subseteq\nabla=\nabla_1+\nabla_2$};

\end{tikzpicture}
\caption{Generalized nef partitions $\Pi(\Delta)$ and $\Pi(\nabla)$}
\label{fig113dual}

\end{figure}

The same can be done for the generalized nef partition $\Pi(\Delta)=\{\Delta_1,\Delta_2,\Delta_3\}$ associated to the partition in \eqref{eqs=3}. In this case, one obtains 
\[\nabla_1=\Conv((0,0), ((-1,-1)),\ 
\nabla_2=\Conv((0,0), (1,0)),\
\nabla_3=\Conv((0,0),(-2,1))
\]
In this case, $\nabla=\sum_{i=1}^3\nabla_i$ has vertices $(1,0), (-1,1), (-2,1), (-3,0), (-1,1), (0,-1)$ and it is not reflexive. 
\end{example}

\subsection{Irreducibility}\label{irr}

\begin{definition}\label{def-reducible}
 A generalized nef-partition $\Pi(\Delta)=\{\Delta_1,\dots,\Delta_s\}$ is \emph{reducible} if there exists a proper subset $ A$ of $\{1,\dots,s\}$  such that the Minkowski sum $\sum_{i\in A} \Delta_{i}$ contains $0$ in its relative interior, or equivalently if
 \[
\ell^*\left(\sum_{i\in A} \Delta_{i}\right)=1.
 \]
 Nef-partitions which are not reducible are called \emph{irreducible}.
\end{definition}

\begin{lemma} \label{4.12}
Let $\{1,\dots,s\}=A\cup B$ with $A\cap B=\emptyset$ and $A,B$ not empty.
Let $C_A$ be the set of all non-negative real linear combinations of vertices of $\{\Delta_i\}_{i\in A}$. Define similarly $C_B$. Then $C_A\cap C_B=\{0\}$.

In particular, if the Minkowski sum $\sum_{i\in A}\Delta_i$ contains the origin in its interior and $V_A$ is the linear subspace generated by $\{\Delta_i\}_{i\in A}$, then $V_A\cap C_B=\{0\}$.

\end{lemma}
\begin{proof}
Assume that $m\in C_A\cap C_B$. Since $m\in C_A$, then by the definition of the $\Delta_i$'s we have that 
$(m,n_{\rho})\geq 0$ for all $n_{\rho}\in \Sigma(1)\backslash \bigcup_{k\in A} I_k$. Similarly, since, $m\in C_B$, then $(m,n_{\rho})\geq 0$ for all $n_{\rho}\in \Sigma(1)\backslash \bigcup_{k\in B} I_k$. Thus $(m,n_{\rho})\geq 0$ for all $n_{\rho}\in \Sigma(1)$. Since $\Delta^\circ$ contains the origin in its interior, this implies $m=0$.

The last assertion follows from the fact that 
if $\sum_{i\in A}\Delta_i$ contains the origin in its interior, then $C_A=V_A$.
\end{proof}
We will now generalize the result in \cite[Thm 5.8]{BB} to generalized nef partitions. The aim of Theorem \ref{thm_irred} is to prove that, for a generalized nef partition, being reducible according to Definition \ref{def-reducible} is equivalent to admit a decomposition into direct sum of irreducible generalized nef partitions.

\begin{definition}\label{directsum}
A generalized nef partition $\Pi(\Delta)=\{\Delta_1,\dots, \Delta_s\}$ is a {\em direct sum of generalized nef partitions} $\Pi(\Delta^{(1)}),\dots,\Pi(\Delta^{(k)})$ if  
$\Delta^{(1)},\dots,\Delta^{(k)}$ are rational
polytopes contained in $\Delta$ whose linear spans $M_1,\dots,M_k$ satisfy $M_{\rr}=M_1\oplus \cdots\oplus M_k$,
$\Delta=\Delta^{(1)}+\cdots+\Delta^{(k)}$
and such that 
\[
\Pi(\Delta^{(i)})=\{\Delta_j: \Delta_j\subseteq \Delta^{(i)}\}
\]
for any $i\in \{1,\dots,k\}$.
\end{definition}

\begin{theorem}\label{thm_irred}
Any generalized nef partition $\Pi(\Delta)$, where $\Delta$ is a lattice polytope, has a unique decomposition into direct sum of irreducible generalized nef-partitions.
\end{theorem}
    
\begin{proof}
Assume that $\Pi(\Delta)$ is reducible. 
Up to reordering we can assume that  $\Delta':=\Delta_{1}+\cdots+ \Delta_{k}$, $k<s$, contains the origin in its relative interior. 

Let $C'$ be the set of all non-negative real linear combinations of vertices of $\Delta_1,\dots,\Delta_k$ and $V'$ be the linear subspace of $M_{\rr}$ generated by $C'$. Since the origin is in the interior of $\Delta_{1}+\cdots+ \Delta_{k}$, then $C'=V'$.

Let $C''$ be the set of all non-negative real linear combinations of vertices of $\Delta_{k+1},\dots,\Delta_s$ and $V''$ be its linear span.
By Lemma \ref{4.12} $V'\cap C''=\{0\}$.
We now show that $C'$ contains the origin in its interior. If this is not the case, then by the separatness theorem for convex sets, there exists $y\in N_{\rr}$ non-zero such that $(x,y)=0$ for all $x\in V'$ and $(x,y)\geq 0$ for all $x\in C''$. Thus $(x,v)\geq 0$ for any $v\in \Delta_1\cup\cdots \cup\Delta_s$, contradicting the fact that $\Delta$ contains the origin in its interior.
This implies that $C''=V''$ and $M_\rr=V'\oplus V''$. 
Observe that the same argument shows that 
the origin is in the interior of $\Delta'':=\Delta_{k+1}+\cdots+\Delta_{s}$.

Under the isomorphism $N_{\rr}\cong (V')^\vee \oplus (V'')^\vee$  we have that 
$\Delta^\circ={\rm Conv}((\Delta')^\circ,(\Delta'')^\circ)$, in particular
the vertices of $\Delta^\circ$ are either  vertices of $(\Delta')^\circ$ or vertices of $(\Delta'')^\circ$. Thus $I_1\cup\cdots\cup I_k$ is the set of vertices of $(\Delta')^\circ$ and $I_{k+1}\cup\cdots\cup I_s$ is the set of vertices of $(\Delta'')^\circ$. 
It is now easy to check that $I_1\cup\cdots\cup I_k$ defines a generalized nef partition in $V'$ with associated polytopes $\Delta_1,\dots,\Delta_k$ and analogously for $I_{k+1}\cup\cdots\cup I_s$.

The statement thus follows by induction.
\end{proof}

Duality preserves irreducibility, as proven in the following Lemma.

\begin{lemma}\label{dual-irred}
    If $\Pi(\Delta)$ is irreducible then $\Pi(\nabla)$ is irreducible.
\end{lemma}

\begin{proof}
    Assume $\nabla=\nabla^{(1)}\oplus\nabla^{(2)}$ decomposed as the direct sum of two generalized nef partitions, where $\nabla^{(1)}=\nabla_1+\ldots+\nabla_d$ and $\nabla^{(2)}=\nabla_{d+1}+\ldots+\nabla_s$
    for some $d$.
    Each $\nabla^{(i)}$ generates a subspace $M^{(i)}$ and $M^{(1)}\cap M^{(2)}=\{0\}$. 
    
We want to prove that $\Delta$ decomposes as the sum of two  generalized nef partitions.
Let ${\Delta^{(i)}}^\circ\subset M^{(i)}, i=1,2$ be 
$${\Delta^{(1)}}^\circ:={\rm Conv}(\nabla_1,\ldots,\nabla_d),\ {\Delta^{(2)}}^\circ:={\rm Conv}(\nabla_{d+1},\ldots,\nabla_s).$$

By Lemma \ref{lemma-part}, $\Delta^\circ={\rm Conv}(\nabla_1,\ldots\nabla_s)$.
Clearly $ {\Delta^{(i)}}^\circ\subset\Delta^\circ$ and $\Delta\subset \Delta^{(1)}\oplus N^{(2)}$ and 
$\Delta\subset \Delta^{(2)}\oplus N^{(1)}$.
Thus $\Delta\subset\Delta^{(1)}+\Delta^{(2)}$.

For the other inclusion let $\alpha+\beta\in\Delta^{(1)}+\Delta^{(2)}$. We need to prove that $\alpha+\beta\in\left({\Conv}(\nabla_1,\ldots, \nabla_s)\right)^\circ=\bigcap\nabla_i^\circ$.
Given $x\in \nabla_i$, we study $\langle\alpha+\beta,x\rangle$.
If $x\in\nabla^{(1)}$, 
$\langle\alpha,x\rangle\geq -1,$ and $\langle\beta,x\rangle=0$ since $\beta\in(M^{(2)})^\vee$ and $x\in M^{(1)}$.  
Similarly if $x\in\nabla^{(2)}$, 
$\langle\alpha,x\rangle= 0,$ and $\langle\beta,x\rangle\geq -1$. 
\end{proof}

\subsection{Toric geometry of generalized nef partitions}\label{sec-geometry}
In this section we will give a geometric interpretation of 
generalized nef partitions under some extra hypotheses on the polytopes. 
We recall that a lattice polytope in $N_\qq$ containing the origin in its interior is $\mathbb Q$-Fano if its vertices are primitive in $N$. Equivalently, the toric variety defined by the normal fan of its polar is $\qq$-Fano, see \cite[Theorem 6.2.1]{CLS}, \cite[Thm.1.3]{ACG}.

\begin{proposition}\label{geo-gnp}
Let $\Delta$ be a polytope in $M_\qq$ containing the origin in its interior such that $\Delta^\circ$ is $\mathbb Q$-Fano and let $Z=Z_{\Delta}$ be the $\qq$-Fano toric variety defined by the normal fan to $\Delta$.
Giving a generalized nef partition $\Pi(\Delta)=\{\Delta_1,\dots,\Delta_s\}$ of $\Delta$ is equivalent to give a partition $R_1\sqcup \cdots \sqcup R_s$ of the set of $1$-dimensional cones of the fan of $Z$ such that the divisors $N_i:=\sum_{\rho \in R_i} D_{\rho}$
are $\mathbb Q$-Cartier and nef for any $i=1,\dots,s$. Moreover $\Delta_i=\Delta(N_i)$ for any $i$.
\end{proposition}

\begin{proof} The primitive generators of the rays of the fan of $Z$ are exactly the vertices of $\Delta^\circ$, so that $\Delta$ is the polytope associated to the anticanonical divisor $-K_Z$.
If $\Pi(\Delta)$ 
is a generalized nef partition of $\Delta$ associated to  
${\rm Vert}(\Delta^\circ)=I_1\sqcup\ldots\sqcup I_s$, then the condition in Proposition \ref{car}\,(iv)  exactly means that the divisors $N_i$ are $\qq$-Cartier and nef \cite[Theorem 6.1.7 and Lemma 9.2.1]{CLS}.
The last assertion is clear from the definitions of $\Delta_i$ and $\Delta(N_i)$. 
 \end{proof}

\begin{remark}\label{rmk_geometry}
Observe that, to give geometric meaning  as in Proposition \ref{geo-gnp} to a generalized nef partition $\Pi(\Delta)$ and to its dual $\Pi(\nabla)$ one needs both $\Delta^\circ$ and $\nabla^\circ$ to be lattice polytopes. The latter implies that the $\Delta_i$'s are lattice polytopes by Lemma \ref{4.12}.
Thus $\Delta$ should be reflexive and $\Pi(\Delta)$, $\Pi(\nabla)$ are dual nef partitions in the classical sense, see \cite[Def. 4.6]{BB} and \cite[Def. 2.5]{Bor1}.

\end{remark}

\begin{example}[{\bf Continuation of Example \ref{ex-part}}]
Let $\Delta$ be as in Example \ref{ex-part}.  
The vertices of $\Delta^\circ$ are the minimal generators of the rays of the fan of $Z$:
 \[
n_1=(0,-1),\ n_2=(-1,1),\ n_3=(1,0),\ n_4=(0,1).
 \] We already observed that $Z=Z_\Delta$ is the blow up of $\pp^2$ at one point
and that the partition  ${\rm Vert} (\Delta^\circ)=\{n_2, n_4\}\sqcup\{n_1,n_3\}$ defines a generalized nef partition.

Let $D_1,\ldots,D_4$ be the prime toric divisors corresponding to  $n_1,\dots,n_4$ respectively. 
Fixed the natural basis $\{e_0,e_1\}$ of $\Cl(Z)$ ($e_0^2=1, e_1^2=-1, e_0\cdot e_1=1$), the classes of $D_1,\dots, D_4$ are respectively 
$e_0, e_0-e_1, e_0-e_1, e_1$.
The partition $I_1\sqcup I_2$ defines the divisors $D_2+D_4, D_1+D_3$ whose classes $e_0$ and $2e_0-e_1$ are both nef.

As observed, the partition $\{n_2,n_3\}\sqcup\{n_1,n_4\}$ does not give a generalized nef partition and in fact in this case it defines the divisors $D_2+D_3, D_1+D_4$, and one can observe that the class of the divisor $D_1+D_4$ is $e_0+e_1$, which is not nef.

In Table \ref{tab-classes} we show the nef classes for the other generalized nef partitions of Example \ref{ex-part}.
\begin{table}[h]
    \centering
    \begin{tabular}{c|c}
    Partition&classes\\
    \hline\hline
        $\{n_2,n_4\}\sqcup\{n_1,n_3\}$ & $e_0, 2e_0-e_1$ \\\hline
        $\{n_1\}\sqcup\{n_2,n_3,n_4\}$& $e_0, 2e_0-e_1$\\\hline
         $\{n_2\}\sqcup\{n_1,n_3,n_4\}$& $e_0-e_1, 2e_0$\\[3pt]
    \end{tabular}
    \caption{Nef classes}
    \label{tab-classes}
\end{table}

Let $\Pi(\Delta)=\{\Delta_1,\Delta_2\}$ be the nef partition associated to $I_1\sqcup I_2=\{n_2, n_4\}\sqcup\{n_1,n_3\}$. 
By Definition \ref{def-nabla} one gets
$$\nabla_{1} = \text{Conv}((0,0), n_2, n_4)\mbox{ and } \nabla_{2} = \text{Conv}((0,0), n_1, n_3),$$ 
and 
$\nabla = \nabla_{1} + \nabla_{2} = \text{Conv}((-1, 1), (-1, 0), (0, -1), (1, 0), (1, 1))$,
(see Figure \ref{fig:GNP}).
The dual partition $J_1\sqcup J_2$ of vertices of $\nabla^\circ$ is 
\[
J_1\sqcup J_2=\{(0,-1),(1,0)\}\sqcup\{(-1,1),(-1,0), (1,1)\}.
\]

\begin{figure}
    \centering
    \includegraphics[width=0.5\linewidth]{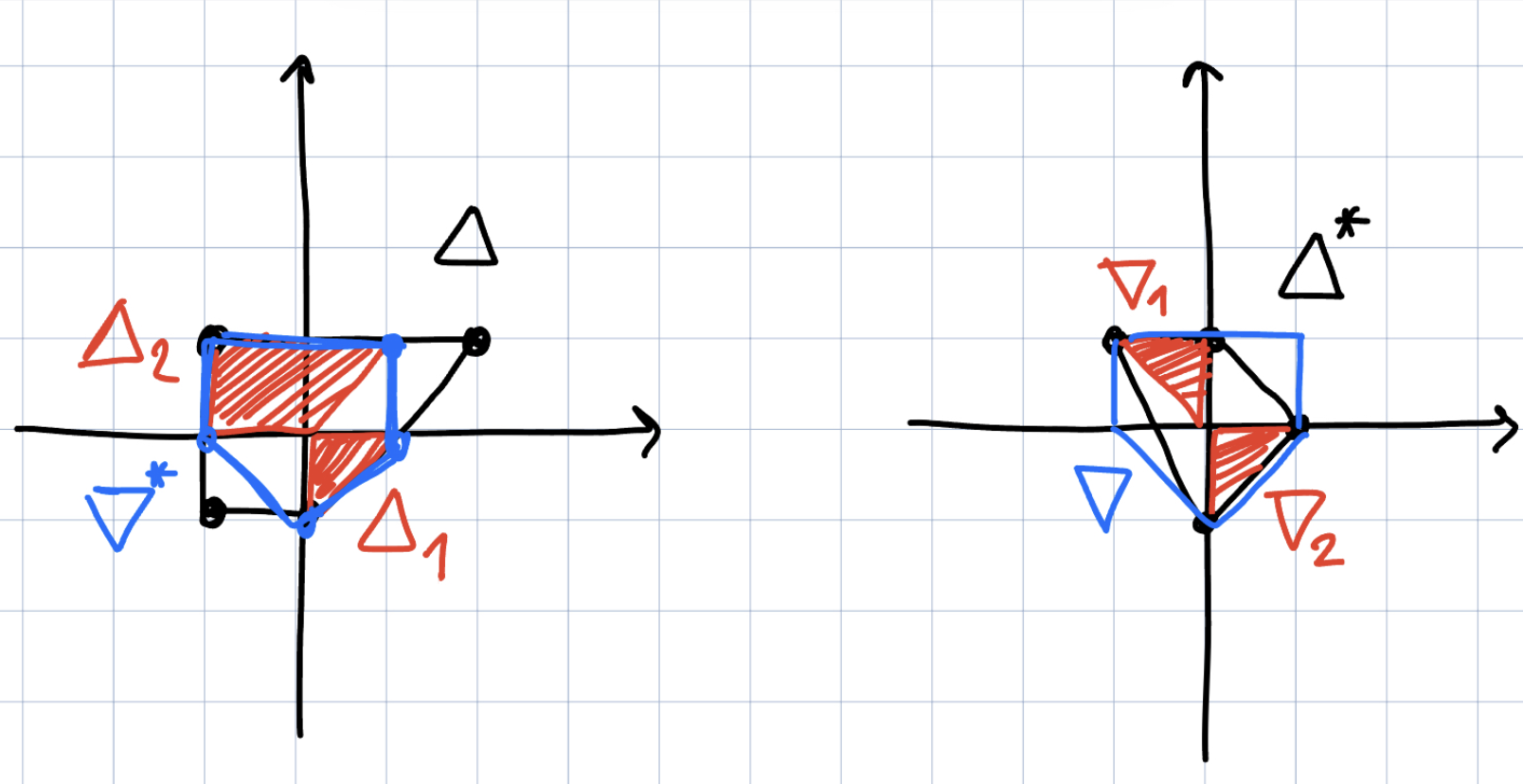}
    \caption{Dual nef partitions $\Pi(\Delta) $ and $\Pi(\nabla)$}
    \label{fig:GNP}
\end{figure}

 In this particular case, $\Delta$ is reflexive so according to Remark \ref{rmk_geometry}, it makes sense to consider the variety associated to the dual partition $\Pi(\nabla)$.
 We first observe that the toric variety $Z^*$ associated to $\nabla$ is the blow up of $\pp(1,1,2)$ at the two smooth torus-fixed points, and the rays of its fan are generated by 
\[
n_1^*=(-1, 1),\ n_2^*=(-1, 0),\ n_3^*=(0, -1),\ n_4^*=(1, 0),\ n_5^*=(1, 1). \] 
 In coordinates, the toric divisors $D_2$ and $D_4$ associated to the rays $n_2^*, n_4^*$ can be assumed to correspond to the two exceptional divisors of the blow-up 
over $(1,0,0)$ and $(0,1,0)$, 
while the divisors corresponding to $n_1^*, n_3^*, n_5^*$ are the proper transforms of the coordinate divisors $y=0$, $z=0$ and $x=0$ of $\mathbb P(1,1,2)$ respectively.
The partition $J_1\sqcup J_2$ defines the divisors $D_3+D_4$ and $D_1+D_2+D_5$. 
The corresponding classes, with respect to the basis $\{e_0,e_1,e_2\}$ of $\Cl(Z^*)$, where $e_0$ is the class of the pull-back of the line $x=0$ and $e_1,e_2$ are the classes of the exceptional divisors,  are $2e_0-e_2, 2e_0-e_1$ which are both nef.   

A similar analysis can be done for the generalized nef partitions associated to $\{n_1\}\sqcup\{n_2,n_3,n_4\}$ and $\{n_2\}\sqcup\{n_1,n_3,n_4\}$ of Example \ref{ex-part}, obtaining the varieties of Table \ref{tab:dual}, where the $p_i$'s denote torus-invariant points and $\iota$ is induced by the toric involution of $\pp^2$ which fixes a line through $p$ (see \cite{2dpolygons}). We denote by  $e_0$ the pull-back of the class of a line in $\pp^2$, $e_i$ the class of the exceptional divisor over the point $p_i$, $e$ is the class of the image of the exceptional divisor over $p$ in the quotient ${\rm Bl}_p\pp^2/\langle \iota\rangle$ and $f$ the class of the image of the proper transform of the line through $p$ and the isolated fixed point of $\iota$.

\begin{table}[h]
    \centering
    \begin{tabular}{c|c|c}
         Partition of  ${\rm Vert}(\Delta^\circ)$ &$Z^*$ &Dual partition  of   ${\rm Vert}(\nabla^\circ)$\\
         and classes of divisor &&and classes of divisor\\
        \hline \hline 
        $\{n_2,n_4\}\sqcup\{n_1,n_3\}$&   ${\rm Bl}_{p_1,p_2}\pp(1,1,2)$  & $J_1=\{(0,-1),(1,0)\}$\\
        
      &  &  $J_2=\{(-1,0), (-1,1),(1,1)\}$ \\
    $e_0, 2e_0-e_1$&&$2e_0-e_2, 2e_0-e_1$ \\
    \hline
        $\{n_1\}\sqcup\{n_2,n_3,n_4\}$&${\rm Bl}_{p_1,p_2,p_3}\pp^2$ & $J_1=\{(0, 1),
        (1, 1),\}$\\
              &  &  $J_2=\{(-1, -1),
        (-1,  0),
        ( 0, -1),  ( 1,  0)\}$ \\
      $e_0, 2e_0-e_1$& &  $e_0-e_1, 2e_0-e_2-e_3$\\
    \hline
        $\{n_2\}\sqcup\{n_1,n_3,n_4\}$& ${\rm Bl}_{p}\pp^2/\langle \iota\rangle$ & $J_1=\{(1,0)\}$\\
              &  &  $J_2=\{(-1, -1),
        (-1,  1),
        ( 1,  1)\}$ \\
     $e_0-e_1, 2e_0$& & $2f, 2(f+e)$  
    \end{tabular}
    \caption{Dual partitions and varieties for all generalized nef partitions of Example \ref{ex-part}}
    \label{tab:dual}
\end{table}
\end{example}

\section{Good pairs and Calabi-Yau complete intersections}\label{sec_gp}

In this section, we will define a notion 
of good pairs of generalized nef partitions and a duality between them which generalizes the duality of nef partitions by Batyrev and Borisov \cite{BB} and the duality of good pairs of polytopes given in \cite{ACG}.

\subsection{Good pairs of generalized nef partitions}
A {\it good pair} of polytopes is a pair of polytopes $(\Delta^1, \Delta^2)$ in $M_{\qq}$ such that $\Delta^1\subseteq\Delta^2$ and both $\Delta^1$ and $(\Delta^2)^\circ$ are canonical. By \cite[Corollary 1.6]{ACG} this is equivalent to say that both $\Delta^1$ and $(\Delta^2)^\circ$ are lattice polytopes containing the origin in their interior.

We give a similar definition when both $\Delta^1$ and $\Delta^2$ admit generalized nef partitions.

\begin{definition}
A {\em good pair of 
generalized nef partitions} 
is a pair of generalized nef partitions $(\Pi(\Delta^1),\Pi(\Delta^2))$, where  $\Pi(\Delta^1)=\{\Delta^1_1,\ldots,\Delta^1_s\}$ and $\Pi(\Delta^2)=\{\Delta^2_1,\ldots,\Delta^2_s\}$, 
such that 
\begin{itemize}
    \item $\Delta^1_i$ is a lattice polytope contained in $\Delta_i^2$ for all $i\in\{1,\ldots,s\}$;
    \item $(\Delta^1,\Delta^2)$ is a good pair of polytopes.
\end{itemize} 
A good pair is called {\em irreducible} if $\Pi(\Delta^1)$ and $\Pi(\Delta^2)$ are irreducible.
\end{definition}
 
We now observe that the duality of generalized nef partitions defined in \S\ref{duality} preserves good pairs.

\begin{proposition}\label{goodpairs}
If $\mathcal P=(\Pi(\Delta^1),\Pi(\Delta^2))$ is a good pair of generalized nef partitions, then $\mathcal P^*:=(\Pi(\nabla^2),\Pi(\nabla^1))$ is a good pair of generalized nef partitions. 
\end{proposition}

\begin{proof}
By \cite[Corollary 1.6]{ACG}, it is enough to prove that $\nabla^2$ and $(\nabla^1)^\circ$ are lattice polytopes with the origin as an interior point. By Proposition \ref{prop:origin}, part (iii), we just have to see that these polytopes are lattice polytopes. First, since $(\Delta^2)^\circ$ is a lattice polytope, the same holds for the $\nabla_{i}^2$'s by Proposition \ref{prop:origin}, part (i), and hence $\nabla^2 = \nabla^2_{1} + \ldots + \nabla^2_{s}$ is a lattice polytope. Secondly, since the $\Delta^1_{i}$'s are lattice polytopes, by Lemma \ref{lemma-part}, the same holds for $(\nabla^1)^\circ$. 
\end{proof}

\begin{definition}\label{dualgp}
We will say that $\mathcal P=(\Pi(\Delta^1),\Pi(\Delta^2))$ and  $\mathcal P^*:=(\Pi(\nabla^2),\Pi(\nabla^1))$ are {\em dual good pairs of generalized nef partitions}.
\end{definition}
Clearly $(\mathcal P^*)^*=\mathcal P$, thus we have defined an involution 
in the set of good pairs of generalized nef partitions.

\begin{example}\label{ex-nestedpart}
Let $\Delta^2 := \Delta$ be as in Example \ref{ex-part} 
with the partition of $\text{Vert}((\Delta^2)^\circ)$ given by $I_1 = \{(-1,1),(0,1)\}$, $I_2 = \{(0,-1),(1,0)\}$. 
In Example \ref{ex-part} we computed the polytopes $\Delta_1^2,\Delta_2^2$ and $\nabla_1^2, \nabla_2^2$ such that $\Delta_1^2+\Delta_2^2=\Delta^2$ and 
$\nabla^2=\nabla_1^2+\nabla_2^2$.

Now let us define the polytopes $\Delta_{2}^1 = \rm{Conv}((0,0), (-1, 0), (0, 1))$ contained in $\Delta_{2}^2$, and  $\Delta_{1}^1 = \text{Conv}((0,0), (0, -1),(1, 0))$, which coincides with $\Delta_{1}^2$, as in Figure \ref{fig:nested}. The polytope $\Delta^1=\Delta^1_1+\Delta^1_2$ has vertices $\{(1,0),(0,1),(-1,0),(0,-1)\}$ and $\{\Delta^1_1,\Delta^1_2\}$ is an irreducible   generalized nef partition of $\Delta^1$.
Note that the vertices of $\Delta^1$ are the minimal generators of the fan of $Z^* = X_{\nabla^1} = \mathbb{P}^1 \times \mathbb{P}^1$. 
The corresponding $\nabla_{i}^1$'s are:
\begin{equation}\label{eqnabla}
\nabla_{1}^1 = \text{Conv}(0, (-1,1), (-1,0), (0,1)),\  \nabla_{2}^1 = \text{Conv}(0, (0,-1), (1,-1), (1,0)).
\end{equation}
The pairs $(\Pi(\Delta^1),\Pi(\Delta^2))$ and $(\Pi(\nabla^2),\Pi(\nabla^1))$ are good pairs of generalized nef partitions.

\begin{figure}
    \centering
    \includegraphics[width=0.5\linewidth]{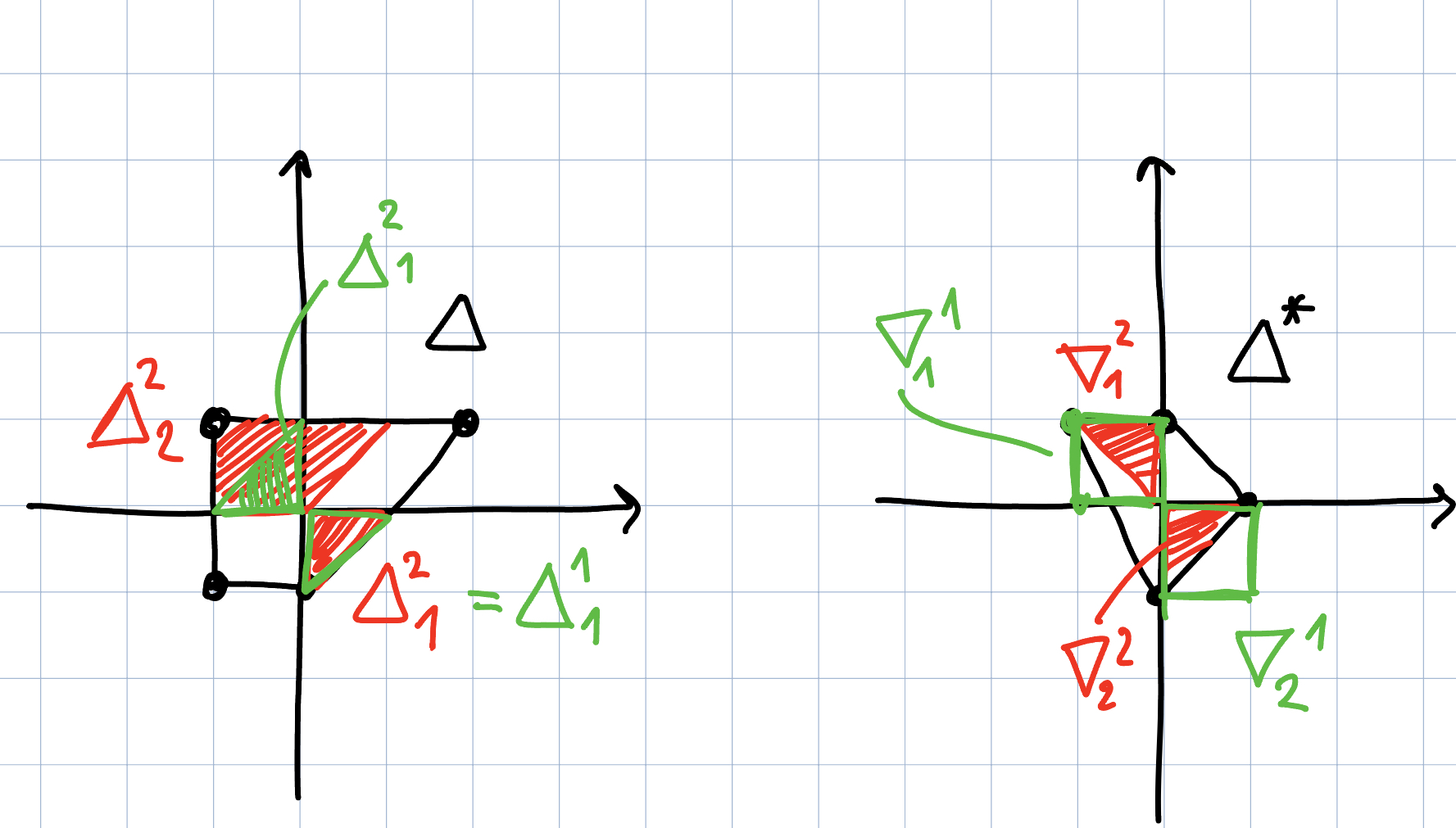}
    \caption{The pair $(\Delta^1,\Delta^2)$}
    \label{fig:nested}
\end{figure}
\end{example}

\subsection{Calabi-Yau varieties associated to good pairs}\label{sec-CY}

In this paper by {\em Calabi-Yau variety} we mean an $n$-dimensional normal projective variety $X$
with canonical singularities, such that $K_X\sim 0$ and $h^i(X,\mathcal O_X) = 0$ for $0 < i < n$.

Let $Z=Z_{\Sigma}$ be a projective toric variety of dimension $n$ and let $X$ be a subvariety, defined by $g_1,\dots,g_s\in \mathcal R(Z)$  
of degrees $d_1,\ldots,d_s\in\Cl(Z)$. 
Let $\Delta_1,\dots,\Delta_s\subseteq M_{\qq}$ be the Newton polytopes of the $\Sigma$-dehomogenizations $f_1,\dots,f_s$ of $g_1,\dots,g_s$.
We will now consider the following assumptions.

\begin{assumption}\label{hypothesis}
\ \\
\vspace{-0.4cm}

\begin{itemize}
\item   $X$ is a quasismooth complete intersection as in Definition \ref{qsci};
\item $g_i$ are general in the sense of \S\ref{2.4};
\item $\sum_i d_i=[-K_Z]$; 
\item 
$\dim(\Delta_i)>0$, 
$n=\dim(\Delta_1+\ldots+\Delta_s)$, $\ell^*(\Delta_1+\ldots+\Delta_s)=1$ and for any proper subset $\{k_1,\ldots,k_a\}\subset\{1,\ldots,s\}$ one has $\ell^*(\Delta_{k_1}+\ldots+\Delta_{k_a})=0$.
\end{itemize}
\end{assumption}

\begin{theorem}\label{thmCY} 
Let $Z$ be a projective toric variety of dimension $n$ 
and let $X\subseteq Z$ be a quasismooth complete intersection, defined by homogeneous elements $g_1,\dots,g_s\in \mathcal R(Z)$  
of degrees $d_1,\ldots,d_s\in\Cl(Z)$ 
satisfying Assumption 1 and with $n\geq s+1$.
Then $X$ is Calabi-Yau variety.
\end{theorem}

\begin{proof}
Let $p:\hat Z\to Z$ be the characteristic morphism of $Z$, which is a good quotient by the action of  $H={\rm Spec}\ \cc[\Cl(Z)]$, and let $\hat X=p^{-1}(X)$.
Since $X$ is quasismooth, then $\hat X$ is smooth, in particular it is normal.  
Moreover $p_{|\hat X} : \hat X \to X$ is a good quotient for the induced action of $H$. Since $\hat X$ is normal, it follows that $X$ is normal (see for example \cite[Lemma 5.0.4]{CLS}).

The intersection of $X$ with the smooth locus of $Z$ is smooth by the quasismothness hypothesis.
Applying the classical adjunction formula to this 
intersection we obtain that $K_X\sim 0$ by Assumption \ref{hypothesis}. In particular $X$ has Gorenstein singularities. Moreover, since $\hat X$ is smooth, the singularities of $X$ are rational \cite[Corollaire]{Bo}. By \cite[Corollary 11.13]{Ko2} Gorenstein rational singularities are canonical.

We now prove that $h^{i,0}(X,\Osh_X)=0$ for $0<i<\dim(X)$.
Let $X_T$ be intersection of $X$ with the torus $T\subset Z$. 
Let $\overline{X}_T$ be the intersection of the closures of the hypersurfaces defining $X_T$ in the toric variety $\mathbb P_{\Delta}$ defined by the normal fan to   $\Delta=\Delta_1+\dots+\Delta_s$. 
By \cite[Corollary 3.5]{BB} $\bar X_T$
is a non-empty irreducible variety of dimension $n-s\geq 1$ with $h^i(\mathcal O_{\bar X_T})=0$ for $i=1,\dots, n-s-1$. 
Observe that $\overline X_T$ is quasismooth in $\mathbb P_{\Delta}$ since the $\Delta_i$'s are supporting polytopes of base point free divisors by \cite[Proposition 2.4]{BB}. Thus $\bar X_T$ has rational singularities.
Moreover $X$ and $\bar X_T$ are birational, since 
$X\cap T=\bar X_T\cap T$.
The result now follows from  Proposition \ref{hi}.
\end{proof}

\begin{proposition}\label{hi}
Let $X,X'$ be two birational complex projective varieties with rational singularities. Then $H^i(X,\mathcal O_X)\cong H^i(X',\mathcal O_{X'})$ for any $i$.
\end{proposition}

\begin{proof}
Since $X$ and $X'$ have rational singularities there exist two birational morphisms $\pi:\tilde X\to X$ and $\pi':\tilde{X'}\to X'$, where $\tilde X,\tilde X'$ are smooth and such that $R^i\pi_*\mathcal O_{\tilde X}=R^i\pi'_*\mathcal O_{\tilde X'}=0$ for $i>0$.
By \cite[Exercise III. 8.1]{Ha} we have that 
$h^i(X,\mathcal O_X)= h^i(\tilde X,\mathcal O_{\tilde X})$ and $h^i(X',\mathcal O_{X'})= h^i(\tilde X',\mathcal O_{\tilde X'})$. 

Since $\tilde X$ and $\tilde X'$ are smooth birational projective varieties, then 
$h^i(\tilde X,\mathcal O_{X})=h^i(\tilde X',\mathcal O_{\tilde X'})$ by \cite[Section 4.2, p. 494]{GH}.
\end{proof}

We will now show that a good pair $\mathcal P$ of irreducible generalized nef partitions, under an extra assumption, defines a family $\mathcal F_\mathcal P$ of Calabi-Yau varieties embedded 
in a $\mathbb Q$-Fano toric variety.

Consider a good pair of  
generalized nef partitions
 $\mathcal P=(\Pi(\Delta^1),\Pi(\Delta^2))$, where  $\Pi(\Delta^1)=\{\Delta^1_1,\ldots,\Delta^1_s\}$ and $\Pi(\Delta^2)=\{\Delta^2_1,\ldots,\Delta^2_s\}$.
Let $Z=Z_{\Delta^2}$ be the toric variety defined by the normal fan to $\Delta^2$. Since $(\Delta^1,\Delta^2)$ is a good pair, the polytope $(\Delta^2)^\circ$  is canonical, thus $Z$ is $\qq$-Fano with canonical singularities. 
The generalized nef partition $\Pi(\Delta^2)$ provides a decomposition 
$-K_Z=N_1+\cdots+N_s$, where $N_i$ are toric $\qq$-Cartier nef divisors with associated polytopes $\Delta^2_1,\dots,\Delta_s^2$. 
For each $i$, the polytope $\Delta^1_i \subseteq\Delta_i^2$  identifies a subspace $\mathcal F(\Delta^1_i)$ of the complete linear system $|N_i|$ on $Z$ (see section \S\ref{subsec-polytopes}). 
Let $X$ be the subvariety of $Z$ defined by $s$ general divisors in $\mathcal F(\Delta^1_1),\dots, \mathcal F(\Delta^1_s)$.
The good pair $\mathcal P$ is defined to be {\em quasismooth} if $X$ is a quasismooth complete intersection as in Definition \ref{qsci}. 
All conditions in Assumption 1 are satisfied by $X$ because of the generality, quasismoothness and irreducibility hypotheses. Thus we obtain the following result.

\begin{theorem}\label{CY}
An irreducible and quasismooth good pair of 
generalized nef partitions $\mathcal P=(\Pi(\Delta^1),\Pi(\Delta^2))$ 
 defines a family $\mathcal F_{\mathcal P}$ of Calabi-Yau varieties  in the $\qq$-Fano toric variety $Z_{\Delta^2}$.
\end{theorem}

\begin{notation}\label{coef}
    In order to lighten the notation, in what follows we will omit the coefficients of the monomials of $g_i$. If not otherwise stated, we always consider equations where the coefficients of the given monomials are general in $\mathbb C^*$.
\end{notation}
\begin{example}\label{exK3}
    Let $\mathcal P=(\Pi(\Delta^1),\Pi(\Delta^2))$ be the good pair of generalized nef partitions whose polytopes have the following vertices:\\
   \begin{center}
        \begin{tabular}{c|c|c|c}
             Vertices of $\Delta^2_1$&Vertices of $\Delta^2_2$&Vertices of $\Delta^1_1$&Vertices of $\Delta^1_2$  \\
             \hline

             $ (-1,  0,   -1,  -2)$&$( -1,    1,   -2,  -1)$&$(0,-1,0,-1)$&$(-1,1,-2,3)$\\
             $ ( -1,    0,  -1,    2)$&$  ( -1,    1,   -2,   3)$&$(-1,0,3,-2)$&$(0,1,-1,0)$\\
             $( -1,    0,   3,   -2)$&$( -1,    1,    2,  -1)$&$(1,0,1,0)$&$(0,0,0,0,0)$\\
             
             $(  1,    0,   1,    0)$&$ (  1,    1,    0,   1)$&$(0,0,0,0)$&\\
             $   (\frac{1}{3}, -\frac{4}{3}, \frac{1}{3}, -\frac{2}{3})$&$ (\frac{1}{3}, -\frac{1}{3}, -\frac{2}{3}, \frac{1}{3})$&&\\[10pt]   \end{tabular}
        \label{example3fold}
        \end{center}
The polytope $\Delta^2=\Delta^2_1+\Delta^2_2$ has vertices 
 \begin{align*}
      &( -2,    1,   -3,   -3),
    ( -2,    1,   -3,    5),
    ( -2,    1,    5,   -3),
    (  2,    1,    1,    1),
    (2/3, -5/3, -1/3, -1/3)
 \end{align*} 
and it is the anticanonical polytope of $\pp(1,1,1,2,3)$. 
Observe that $\Delta^2$ is not reflexive.
The vertices of $(\Delta^2)^\circ$ are the rays of the fan of $Z$:
\begin{align*}
      &n_1=(0,1,-1,-1),\ n_2=(-1, 0, 0,1),\ n_3=(-1, 0, 1, 0),\\
      &n_4=(1, 1, 0, 0),\ n_5=(0, -1, 0, 0) \end{align*}
The partition of ${\rm Vert((\Delta^2)^\circ)}$ given by $I_1\sqcup I_2=\{n_1, n_3, n_4\} \sqcup \{n_2, n_5\}$ gives the good pair of generalized nef partitions $\mathcal P$, which is irreducible and quasismooth. 
Therefore, by Theorem \ref{CY}, it defines the family of Calabi-Yau surfaces, i.e. K3 surfaces, in $Z_{\Delta^2}=\pp(1,1,1,2,3)$ with equations
\[\begin{cases}
g_1 = x_3^4 + x_1x_2x_4 + x_4^2 +x_1x_5= 0 \\
g_2 = x_2^4 + x_1^2x_4 + x_3x_5=0,
\end{cases}\]
where we omit the coefficients in the equations  according to Notation \ref{coef}.

A different generalized nef partition  can be obtained  considering $\Pi(\tilde \Delta_1)=(\tilde{\Delta^1_1},\tilde{\Delta^1_2})$, where $\tilde{\Delta^1_i}$ is the convex hull of the lattice points of $\Delta^2_i, i=1,2$. 
This gives a different good pair  $(\Pi(\tilde\Delta^1),\Pi(\Delta_2))$. However, using the  MAGMA functions in Appendix \ref{magma}, it can be proved that this pair is not quasismooth.
\end{example}

\subsection{Duality}\label{subs_duality}
The duality between good pairs of generalized nef partitions   in Definition \ref{dualgp} 
and Theorem \ref{CY} naturally define a 
duality between families of Calabi-Yau varieties associated to good pairs of generalized nef partitions.

More explicitly, if $\mathcal P$ is a good pair of generalized nef partitions as in the statement of Theorem \ref{CY}, then  $\mathcal P^*=(\Pi(\nabla^2),\Pi(\nabla^1))$ is also an irreducible good pair of generalized nef partitions  
by Proposition \ref{dualgnp}, Proposition \ref{goodpairs} and Lemma \ref{dual-irred}. 
If $\mathcal P^*$ is also quasismooth this defines {\em dual families $\mathcal F_{\mathcal P}$ and $\mathcal F_{\mathcal P^*}$  
of Calabi-Yau varieties} in $\qq$-Fano toric varieties.

In what follows we will denote by $Z^*$ the dual toric ambient space, by $X^*$ a Calabi-Yau variety in the dual family $\mathcal F_{\mathcal P^*}$, and by $g_1^*,\dots,g_s^*$ its equations in the Cox ring of $Z^*$.

\begin{remark}
In Remark \ref{rmk_geometry} we observed that, in order to be able to give a certain toric geometric interpretation to both a generalized nef partition and its dual,  we need $\Delta$ to be a lattice polytope. Here the hypothesis of being a good pair ensures this condition on $\Delta^2$ and $\nabla^1$.
\end{remark}

Unfortunately the quasismoothness of $\mathcal P^*$ does not follow from the properties of $\mathcal P$ in general, as the following example shows.
In \S \ref{codim2} we will prove that it holds for all families of quasismooth codimension two K3 surfaces in weigthed projective spaces.
We expect that this can be proven true in a more general setting, we will explore this question in further work. 

\begin{example}\label{counterexample}
    Let $(\Pi(\tilde{\Delta^1}),\Pi(\Delta^2))$ be as in Example \ref{exK3}. We observed that this is a good pair of generalized nef partitions which  is not quasismooth. On the other hand, its dual $(\Pi(\nabla^2),\Pi(\tilde{\nabla^1}))$ is quasismooth; this can be proven with the MAGMA  functions in the Appendix \ref{magma}.
\end{example}

 \begin{definition}\label{matrix}
Let $\mathcal P=(\Pi(\Delta^1),\Pi(\Delta^2))$ be a good pair of generalized nef partitions. Fix an order for the  vertices $m_i^1,\dots, m_i^{k_i}$ of $\Delta_i^1$ and for the  vertices $n_j^1,\dots, n_j^{\ell_j}$ of $\nabla_j^2$, where $i,j\in \{1,\dots,s\}$.
The {\em matrix of $\mathcal P$} is  
\[ 
A_{\mathcal P} = 
\begin{pmatrix}
  A_{1,1} & \cdots & A_{1,s} \\
  \vdots & \ddots & \vdots \\
  A_{s,1} & \cdots & A_{s,s}
\end{pmatrix},
\]
where each block is of the form \(A_{i,j} = (\langle m, n\rangle+  \delta_{i,j})_{m,n}\), with \(m\) running in the list $[0,m_i^1,\dots, m_i^{k_i}]$ and \(n\)  in the list $[0,n_j^1,\dots, n_j^{\ell_j}]$.
\end{definition}

\begin{remark}  
In terms of the defining equations $g_1,\dots,g_s$ of the subvarieties associated to   $\mathcal P$, observe that the blocks 
$A_{i,1},\dots, A_{i,s}$ 
correspond to the polynomial $g_i$. More precisely, each row of $A_{i,j}$ starts with $\delta_{ij}$ and follows with the list of exponents of a monomial of $g_i$ with respect to the variables of $\mathcal R(Z_{\Delta^2})$ corresponding to the elements of $I_j=\{n_j^1,\dots, n_j^{\ell_j}\}$.
Equivalently, $\mathcal A_\mathcal P$ is the matrix of exponents of the polynomial $F=t_1g_1+\cdots +t_sg_s
\in \mathcal R(Z(E))$ given in \eqref{pot}.
\end{remark}

An immediate consequence of the definition is the following.
\begin{proposition}
$A_{\mathcal P^*}=A_{\mathcal P}^T$
\end{proposition}

\begin{example}\label{ex_matrix}
   Let $\mathcal P=(\Pi(\Delta^1),\Pi(\Delta^2))$ be the good pair of generalized nef partitions of Example \ref{ex-nestedpart} and $\mathcal P^*=(\Pi(\nabla^2),\Pi(\nabla^1))$ be its dual. Here $r=4$, $s=2$ and the matrix of the pair is
    \[
    \label{A-ex}
A_{\mathcal P} = \left(
\begin{array}{ccc|ccc}
1 & 1 & 1 & 0 & 0 & 0 \\
1 & 0 & 0 & 0 & 1 & 0 \\
1 & 0 & 1 & 0 & 0 & 1 \\
\hline
0 & 0 & 0 & 1 & 1 & 1 \\
0 & 1 & 0 & 1 & 1 & 0 \\
0 & 1 & 1 & 1 & 0 & 1
\end{array}\right)
\]

 One can show that the good pair $\mathcal P$ gives the complete intersection in ${\rm Bl}_p\pp^2$ with equations  
 \[
 \begin{cases}
g_1 =  x_1x_3 + x_0 + x_2x_3=0 \\
g_2 =  x_0x_1 + x_1x_2x_3 +  x_0x_2=0. 
\end{cases}\] 
 
 We now show how the transposed matrix $A^T=A_{\mathcal P^*}$ gives a complete intersection in $Z^*=X_{\nabla^1}$. 
 First observe that from \eqref{eqnabla}, $\nabla^1\subset N$ is the polytope with vertices 
 \[ (1,1), (1,-1), (-1,1), (-1,-1).
\]
Therefore $Z^*=Z_{\nabla^1}=\pp^1\times\pp^1$, with variables $(y_1,y_2),(y_3,y_4)$.
If we order the columns of the transposed matrix $A^T$ in the following way:
 \[\label{AT-ex}
A_{\mathcal P^*}=A^T = \left(
\begin{array}{ccc|ccc}
1 & 1 & 1 & 0 & 0 & 0 \\
1 & 0 & 0 & 0 & 1 & 1 \\
1 & 0 & 1 & 0 & 0 & 1 \\
\hline
0 & 0 & 0 & 1 & 1 & 1 \\
0 & 1 & 0 & 1 & 1 & 0 \\
0 & 0 & 1 & 1 & 0 & 1\\
\end{array}\right),
\]
we get the dual equations 
 \[
 \begin{cases}
g_1^* =  y_1y_2 +  y_3y_4 +y_2y_4=0  \\
g_2^* =  y_3y_4 + y_1y_3 +y_2y_4=0.
\end{cases}
\]
Observe that in this example the complete intersections $X$ and $X^*$ have dimension 0, i.e. they are pairs of points.
\end{example}
\begin{remark}
  If  $s=1$, this construction recovers the generalized BHK duality described in \cite{ACG}.
\end{remark}

In case the polytopes $\Delta^2$ and $\nabla^1$ are simplices, both ambient spaces $Z$ and $Z^*$ are fake weighted projective spaces, thus we obtain a duality between families of Calabi-Yau varieties  which can be considered as a generalization of the classical Berglund-H\"ubsch-Krawitz duality \cite{BH,Kr, ACG} to the complete intersection case.

\begin{definition}\label{delsarte}
    A good pair of generalized nef partitions $(\Pi(\Delta^1),\Pi(\Delta^2))$ is a {\em Delsarte pair} if $\Delta^2$ and $\Conv(\Delta_1^1,\ldots,\Delta_s^1)$ are simplices.
\end{definition}
Observe that this implies that $\nabla^1$ and $\Conv(\nabla_1^2,\ldots,\nabla_s^2)$ are simplices by Proposition \ref{prop:origin}, i.e. $(\Pi(\nabla^2),\Pi(\nabla^1))$ is also a Delsarte pair.

It follows from Lemma \ref{4.12} that ${\rm Vert}(\Conv(\Delta^1_1,\ldots,\Delta^1_s))=\bigcup_{i}{\rm Vert}(\Delta^1_i)\backslash\{0\}$. Thus if $(\Delta^1,\Delta^2)$ is a Delsarte pair, then 
\begin{equation}\label{eq-delsarte}
\sum_{i=1}^s \#{\rm Vert}(\Delta_i^1)=n+s+1.
\end{equation}
Equivalently, the polynomial $F$ has the same number
of variables as monomials, i.e. $A_\mathcal P$ is a square matrix.

\begin{example}
In Example \ref{exK3} we considered an irreducible and quasismooth good pair of generalized nef partitions $\mathcal P=(\Pi(\Delta^1),\Pi(\Delta^2))$ defining a family of K3 surfaces $\mathcal F_\mathcal P$ in $Z_{\Delta^2}=\pp(1,1,1,2,3)$. The dual pair $\mathcal P^*=(\Pi(\nabla^2),\Pi(\nabla^1))$ is an irreducible good pair of generalized nef partitions and one can prove that it is quasismooth. Thus it defines a dual family of Calabi-Yau surfaces $\mathcal F_{\mathcal P^*}$. 
The dual ambient space $Z^*=Z_{\nabla^1}$ is the fake weighted projective space with integer grading $(1,2,2,3,4)$ and quotient grading $1/3(0,1,2,1,0)$.
If we call $y_1,\ldots,y_5$ the variables of $Z^*$, a general element of the family $\mathcal F_{\mathcal P^*}$ has equations (see Notation \ref{coef})
\[
\begin{cases}
g^*_1=y_1^2y_5 + y_1y_2y_3 + y_3y_5^2 + y_4^4=0\\
g_2^*=y_2^4 + y_3 + y_4y_5=0.
\end{cases}
\]
\end{example}

\subsection{From equations to good pairs of generalized nef partitions}

We will now explain under which assumptions and how, given a family of Calabi-Yau complete intersections in a toric variety, we can apply the duality of good pairs of generalized nef partitions defined in the previous sections.

Let $Z$ be an $n$-dimensional $\mathbb Q$-Fano projective toric variety with canonical singularities with 
anticanonical polytope $\Delta^2$ and
Cox ring $\mathcal R(Z)=\cc[x_1,\dots,x_r]$.
Let $X$ be a quasismooth complete intersection Calabi-Yau, defined as the zero set of the homogeneous elements $g_1,\dots,g_s\in \mathcal R(Z)$ 
whose degrees are $\mathbb Q$-Cartier and nef in $\Cl(Z)$. 

\begin{assumption}\label{assump2} For each $i=1,\ldots,s$, we assume that the polynomial $g_i$ contains a monomial $m_i$ such that 
     \[
     m_1\cdot \ldots\cdot m_s=\prod_{j=1}^rx_j.
     \]
Fix a set of such monomials, 
    which will be called {\em marked monomials}.
    \end{assumption}
    
   The marked monomials give a partition of the variables $x_1,\ldots,x_r$, or equivalently a partition of the rays of the fan of $Z$.  We can then construct the polytopes $\Delta^2_1,\ldots,\Delta^2_s$ as in \eqref{deltai} and these define a generalized nef partition $\Pi(\Delta^2)$ since $\deg(m_i)=\deg(g_i)$ are $\mathbb Q$-Cartier and nef for all $i$ (see Proposition \ref{geo-gnp}).
   
   For each $g_i$ consider the polytope $\Delta^1_i:=\Delta(D_i,g_i)\subseteq \Delta^2_i$,
    where $D_i$ is the toric divisor obtained as the zero set of $m_i$ (see Notation \ref{not1}). Let $\Delta^1:=\Delta^1_1+\ldots+\Delta^1_s$.
    
  \begin{assumption} \label{assump3} Assume that $\Delta^1$ contains the origin in its interior and that $\Pi(\Delta^1)=(\Delta^1_1,\ldots,\Delta^1_s)$ is a generalized nef partition. 
  \end{assumption}
    
    Observe that if $\Delta^1$ is reflexive, then the conditions in Assumption \ref{assump3}  hold, by Remark \ref{rem-gnp} and the fact that $\Delta_1^1\subset\Delta^2_i$ for all $i$. 
    
    Under the previous assumptions  $(\Pi(\Delta^1),\Pi(\Delta^2))$ is a good pair of generalized nef partitions.

We now show how the duality works in several examples. We will see that in 
some known examples the construction gives the dual family already found in literature.

\begin{example}[{\em Libgober-Teitelbaum example}]\label{ex-LT}
    We consider the quasismooth complete intersection, first considered by Libgober and Teitelbaum in \cite{LT93}, given by the following equations in $Z=\pp^5$:
\[\begin{cases}
g_1 =  \underline{x_1x_2x_3} +  x_4^3 +  x_5^3 +  x_6^3 = 0 \\
g_2 =  \underline{x_4x_5x_6} + x_1^3 + x_2^3 + x_3^3 = 0.
\end{cases}\]
Recall that, by Notation \ref{coef}, we omit the coefficients of the monomials.
The polynomials $g_1,g_2$ satisfy Assumptions \ref{assump2} and \ref{assump3}, in particular a choice of  marked monomials is underlined.
Let $\Pi(\Delta^2)=(\Delta^2_1,\Delta^2_2)$ be the generalized nef partition of the anticanonical polytope $\Delta^2$ of $\pp^5$ associated to the choice of the marked monomials.
 Let $\Delta_i^1\subset\Delta^2_i$, $i=1,2$, be the polytopes associated to the polynomials $g_1,g_2$.
It can be checked that $\Pi(\Delta^1)=(\Delta_1^1,\Delta^1_2)$ is a generalized nef partition of $\Delta^1 = \Delta_{1}^1 + \Delta_{2}^1$ and that $\Delta^1$ contains the origin as an interior point. Therefore $\mathcal P:=(\Pi(\Delta^1), \Pi(\Delta^2))$ is a good pair of generalized nef partitions.

A MAGMA computation gives the dual good pair of generalized nef partitions $\mathcal P^*=(\Pi(\nabla^2),\Pi(\nabla^1))$. 
The polytope $\nabla^1 = \nabla_{1}^1 + \nabla_{2}^1$ has $6$ vertices, none of which is a lattice point, but its dual $(\nabla^1)^\circ$ is a canonical polytope. 
Moreover $\Delta^2$ is a reflexive polytope and $\nabla^2$ is reflexive and hence is a lattice polytope.
The polytope $\nabla^1$ is the anticanonical polytope of the dual ambient space $Z^*=X_{\nabla^1}$ and, in agreement with \cite{LT93}, we find that $Z^*\cong \mathbb{P}^5 / (\mathbb{Z}/3\mathbb{Z} \times \mathbb{Z}/3\mathbb{Z} \times \mathbb{Z}/9\mathbb{Z})$, i.e. $Z^*$ is the fake weighted projective space with integer grading $(1,1,1,1,1,1)$ and quotient gradings 
$ 1/3( 0, 2, 1, 1, 1, 1 ),
    1/3( 0, 0, 0, 1, 2, 0 ),
    1/9( 0, 3, 3, 8, 2, 8 )$.

By considering $A_{\mathcal P^*}=A_P^T$, which in this case coincides with $A_\mathcal P$, one gets the equations of the dual family of quasismooth complete intersections in $Z^*$:
\[
\begin{cases}
g_1^*= y_1y_2y_3 +  y_4^3 + y_5^3 + y_6^3=0 \\
g_2^* =  y_4y_5y_6 +  y_1^3 +  y_2^3 + y_3^3=0.
\end{cases}
\]
Observe that in this case the good pair $\mathcal P$ is a Delsarte pair. 
\end{example}

Observe that different choices of the marked monomials for a Calabi-Yau complete intersection can lead to multiple mirrors.

\begin{example}
Let $Z=\pp(1,1,1,1,2)$, $\Delta^2$ be its anticanonical polytope and $X$ be the family of quasismooth complete intersection given by the following equations, considering Notation \ref{coef}:
\[
\begin{cases}
g_1= x_1x_2x_3+x_2^3 + x_4^3 + x_3x_5=0\\
g_2 = x_4x_5+x_1^3 + x_2x_3^2 +  x_1x_2x_4=0,
\end{cases}
\]
In this case there are two different choices of marked monomials and thus partitions of vertices of $\Delta^2$: the choice of marked monomials $x_1x_2x_3$ (resp. $x_3x_5$) in $g_1$ and $x_4x_5$ (resp. $x_1x_2x_4$) in $g_2$ defines the partition $I_1\sqcup I_2$ (resp. $\bar I_1\sqcup \bar I_2$) of vertices of $(\Delta^2)^\circ$. 
This gives two different generalized nef partitions $(\Delta^2_1,\Delta^2_2)$ and $(\bar\Delta^2_1,\bar \Delta^2_2)$ of $\Delta^2$. 
The construction then gives the good pairs of generalized nef partitions 
$\mathcal P=(\Pi(\Delta^1),\Pi(\Delta^2))$ and $\bar{\mathcal P}=(\Pi(\bar \Delta^1),\Pi(\Delta^2))$.
From the dual partitions $\mathcal P^*$  and $\bar{\mathcal P}^*$ one gets two families of K3 surfaces in different ambient spaces $Z^*$ and $\bar Z^*$: toric varieties whose homogeneous coordinate ring has $6$ variables whose respective gradings are given by the columns of the following matrices: 
 \[
\left[
\begin{array}{cccccc}   
0&  1&  0&  0&  0&  1\\
9&  0&  6& 7&  5& 15\end{array}
\right],\quad \left[
\begin{array}{cccccc}   
0&  0&  1&  0&  0&  1\\
7&  5&  0&  9&  6& 12
    \end{array}
\right].
\]
Observe that both toric varieties have a fibration to $\pp(5,6,7,9)$. Moreover, in both cases, one equation of the dual family of K3 surfaces contains a simple variable 
among its monomials, thus  $X^*$ and $\bar X^*$  are isomorphic to K3 hypersurfaces in $\pp(5,6,7,9)$.
\end{example}

\begin{example}[{\em Schoen's type Calabi-Yau}]\label{ex-S}
Consider the family of complete intersections in $Z=\mathbb P^2\times \mathbb P^2\times \mathbb  P^1$ defined by the following equations:
\[
\begin{cases}
g_1=z_0(x_0^3+x_1^3+x_2^3+x_0x_1x_2)+z_1x_0x_1x_2=0\\
g_2=z_0y_0y_1y_2+z_1(y_0^3+y_1^3+y_2^3+y_0y_1y_2)=0,
\end{cases}
\]
where $(x_0,x_1,x_2)$ are the coordinates of the first component of $Z$, $(y_0,y_1,y_2)$ of the second component and $(z_0,z_1)$ of the third component.
This is an example of the Calabi-Yau threefolds constructed in \cite{Schoen} as fiber products of two rational elliptic surfaces. In this case the surfaces are given both by the Hesse pencil of plane cubics.
The choice of the marked monomials $z_0x_0x_1x_2$ in the first equation and $z_1y_0y_1y_2$ in the second equation defines a nef partition of the anticanonical polytope $\Delta^2$ of $Z$.
Let $\Delta_1^1,\Delta_2^1$ the   polytopes associated to $g_1,g_2$ respectively. Since $\Delta^1=\Delta_1^1+\Delta_2^1$ is reflexive, this defines a nef partition $\Pi(\Delta^1)$. 
Thus $\mathcal P=(\Pi(\Delta^1),\Pi(\Delta^2))$ is a good pair of nef partitions. 
A MAGMA computation gives the dual good pair $\mathcal P^*$. 
The corresponding toric variety $Z^*$ is the quotient of $Z$ by the action of the group $\zz/3\zz\oplus \zz/3\zz$ acting as follows:
\[
\begin{array}{l}
(x_0,x_1,x_2),(y_0,y_1,y_2),(z_0,z_1)\mapsto (x_0,\zeta_3 x_1,\zeta_3^2x_2),(y_0,y_1,y_2),(z_0,z_1)\\[10pt]
(x_0,x_1,x_2),(y_0,y_1,y_2),(z_0,z_1)\mapsto (x_0, x_1, x_2),(y_0,\zeta_3 y_1,\zeta_3^2 y_2),(z_0,z_1)
\end{array}
\]
and the dual complete intersection is defined by the same polynomials in $Z^*$: $g_1^*=g_1, g_2^*=g_2$.
Observe that the group action on $Z$ has $6$ fixed subvarieties defined by $x_i=x_j=0$ and $y_i=y_j=0$ for distinct $i,j$. 
These intersect the Calabi-Yau complete intersection in $6$ disjoint smooth curves of genus one. 
By \cite[Proposition 7.1]{HSS} 
this family is exactly the mirror of the family of Schoen's Calabi-Yau threefolds.
In particular it is known that the Hodge numbers of the Calabi-Yau threefolds in both families are $(h^{1,1},h^{2,1})=(19,19)$ \cite[Lemma 2.1]{HSS}.  

Observe that the choice of the marked monomials $z_1x_0x_1x_2$ and $z_0y_0y_1y_2$ gives a different mirror family: a complete intersection of type $(3,3)$ in a quotient of $\mathbb P^5$ by the group $(\zz/3\zz)^{\oplus 3}$.
\end{example}

\begin{example}[{\em Calabrese-Thomas example}]
Let $Z$ be the blow-up of $\mathbb P^5$ at one point and $\mathcal F$ be the family of Calabi-Yau varieties given as smooth complete intersections of two hypersurfaces of $Z$ in the linear system $|3H-2E|$, where $E$ denotes the exceptional divisor and $H$ the pull-back of a hyperplane in $\mathbb P^5$. Let $R$ be the set of primitive generators of the rays of the fan of $Z$:
\[
R=\left\{e_1,e_2,e_3,e_4,e_5, -\sum_i^5e_i,\sum_i^5e_i\right\},
\]
where $e_i$ denotes the $i$-th element of the canonical basis,  
and consider the partition given by:
\[
I_1=\{e_1,e_2,e_3,\sum_i^5e_i\},\quad 
I_2=\{e_4,e_5, -\sum_i^5e_i\}.
\]
Let $\Pi(\Delta)=(\Delta_1,\Delta_2)$ be the corresponding nef partition of the anticanonical polytope $\Delta$ of $Z$.  
The dual good pair is given by $\Pi^*=(\nabla_1,\nabla_2)$ where
\[
\nabla_1={\rm conv}(0,e_1,e_2,e_3,\sum_i^5e_i),
    \quad
\nabla_2={\rm conv}(0,e_4,e_5, -\sum_i^5e_i).    
\]
The homogeneous coordinate ring of the toric variety $Z^*$  defined by the normal fan to the polytope $\nabla:=\nabla_1+\nabla_2$  has $20$ generators and divisor class group of rank $15$. The dual family of Calabi-Yau complete intersections in $Z^*$ has equations of the form:
\[
\begin{array}{l}
\left(\prod_{i=1}^8 x_i\right)x_{17} x_{18}
+ \prod_{i=1}^{10}x_{2i} 
+ x_5^{2} x_6^{3} x_{13}^{2} x_{14}^{3}
+ x_7^{2} x_8^{3} x_{15}^{2} x_{16}^{3}
+ x_{17}^{2} x_{18}^{3} x_{19}^{2} x_{20}^{3}=0
\\[10pt]
\left(\prod_{i=9}^{16} x_{i}\right)x_{19}x_{20}+
 \prod_{i=1}^{10} x_{2i-1}+
x_1^{2} x_2^{3} x_9^{2} x_{10}^{3}
+ x_3^{2} x_4^{3} x_{11}^{2} x_{12}^{3}=0.
\end{array}
\]
The polytopes associated to the choice of the marked monomials 
$\left(\prod_{i=1}^8 x_i\right)x_{17} x_{18}$ and $\left(\prod_{i=9}^{16} x_{i}\right)x_{19}x_{20}$ are exactly 
$\nabla_1, \nabla_2$.
On the other hand, choosing $\prod_{i=1}^{10}x_{2i}$ and  $\prod_{i=1}^{10} x_{2i-1}$ as marked  monomials one obtains a new nef partition $\Pi^b=(\nabla_1^b,\nabla_2^b)$ of $\nabla$, where
\[
\nabla_1^b={\rm conv}(0,-\sum_{i=1}^5  e_i, -\sum_{i=1}^5  e_i+e_1, -\sum_{i=1}^5  e_i+e_2, -\sum_{i=1}^5  e_i+e_3),
\]
\[
\nabla_2^b={\rm conv}(0,\sum_{i=1}^5  e_i, \sum_{i=1}^5  e_i+e_4,\sum_{i=1}^5  e_i+e_5).
\]
The dual nef partition defines the family $\mathcal F'$ of all complete intersections in $\mathbb P^1\times \mathbb P^4$ of a hypersurface of degree $(1,3)$ and a hypersurface of degree $(1,2)$. Thus the families $\mathcal F$ and $\mathcal F'$ are multiple mirrors, i.e. have the same mirror family.

In \cite{CT} Calabrese and Thomas show that general members of the families $\mathcal F,\mathcal F'$ are derived equivalent Calabi-Yau threefolds.  
In \cite{Li} the author shows that these are multiple mirrors and proves that this implies their birationality. 
 \end{example}

\begin{example}
Let $X\subset\pp^5$ be the family of complete intersection of three quadrics with equations (see Notation \ref{coef}):
\[\begin{cases}
g_1=\underline{x_1x_2}+x_3^2+x_4x_5=0\\
g_2=\underline{x_3x_4}+x_1^2+x_5x_2=0\\
g_3=\underline{x_5x_6}+x_2x_4+x_6^2=0, 
\end{cases}
\] where marked monomials are underlined.
This is a family of quasismooth K3 surfaces in $\pp^5$.
The dual family is given by quasismooth
 complete intersections $X^*$ in the fake weighted projective space with integer gradings $(1,1,1,1,1,1)$ and   quotient gradings $ 1/5( 0, 3, 1, 2, 1, 4 )$ with equations
\[
\begin{cases}
g_1^*=y_2y_4+y_5y_6+y_6^2=0\\
g_2^*=y_1^2+y_2y_5+y_3y_4=0\\
g_3^*=y_1y_2+y_3^2+y_4y_5=0.  
\end{cases}
\] 
\end{example}

\subsection{Codimension two K3 surfaces in weighted projective spaces}\label{codim2}
In this section we consider families of  quasismooth K3 surfaces defined by Delsarte good pairs of generalized nef partitions (see Definition \ref{delsarte}) in 4-dimensional weighted projective spaces. As observed after Definition \ref{delsarte}, the dual good pairs are still of Delsarte type. 
By an exhaustive study of all cases, we prove that in this case quasismoothness is preserved under duality, so that the dual good pair still defines a family of codimension two K3 surfaces in fake weighted projective spaces.

\begin{proposition}\label{prop-codim2}
  Let $\mathcal F_{\mathcal P}$ be a family of  K3 surfaces in $\pp(w_1,\ldots,w_5)$ associated to a quasismooth good pair $\mathcal P$ of generalized nef partitions of Delsarte type. 
Then $\mathcal P^*$ is quasismooth and thus $\mathcal F_{\mathcal P^*}$ is a family of codimension two K3 surfaces in a fake weighted projective space.
\end{proposition}

\begin{proof}
By \cite[Section 13.8, Table 2]{F} and  \cite{altinokbrownreid}, there are exactly $84$ possible
vectors $(m,n,w_1,\ldots,w_5)$ such that there exists a quasismooth K3 surface in $\pp(w_1,\dots,w_5)$ defined by equations of degree $m$ and $n$ (up to permuting $m,n$ or $w_1,\dots, w_5$).

In order to prove the result, we use the MAGMA functions contained in the Appendix   to classify and analyze all possible cases for each vector. 
More precisely, for each vector $(m,n,w_1,\ldots,w_5)$, we compute the list of all generalized nef partitions $(\Delta^2_1,\Delta^2_2)$ compatible with it and, for each of them, the possibilities for $\Delta_1^1,\Delta_2^1$ such that $\mathcal P=(\Pi(\Delta^1),\Pi(\Delta^2))$ is a quasismooth and Delsarte good pair of generalized nef partitions. Afterwards, in each case we compute the dual family and check its quasismoothness and irreducibility.
Observe that, in order to be a Delsarte pair, the total number of vertices in $\Delta^1_1$ and $\Delta^1_2$ has to be equal to the number of variables plus the codimension, in this case $7$, by \eqref{eq-delsarte}. 
For example, the possible quasismooth  K3 surfaces associated to the vector $(m,n,w_1,\ldots,w_5)=(2,3,1,1,1,1,1)$, up to reordering variables, are the ones given in Table \ref{tab_P11111}, which also contains
 the dual ambient space $Z^*=Z_{\nabla^1}$, with variables $y_1,\ldots,y_5$, and the dual equations $g_1^*,g_2^*$ (with the convention in Notation \ref{coef}). 
 \end{proof}

{\footnotesize\begin{table}[h]
\begin{tabular}{c|l|l|l}
$Z$ &$g_1,g_2$&$ Z^*$&  $g_1^*,g_2^*$\\
\hline
\hline
$\pp^4$&$x_1x_2+x_4^2 +x_5^2$&$\pp(2,2,2,3,3)$&$ y_4y_5+ y_1^3+y_2^3$\\
&$x_1^3 + x_2^3 + x_3^3 + x_3x_4x_5$& quot: $1/6(0,1,5,3,0)$&$y_1y_2y_3+y_3^3+y_4^2+y_5^2$\\
\hline
$\pp^4$&$  x_1x_2 +x_4^2 +x_5^2$&$\pp(15,9,8,10,12) $&$y_1y_2+y_5^2+y_3^3$\\
&$    x_2^3 + x_3^3 + x_1^2x_5 + x_3x_4x_5$&&$y_3y_4y_5+y_4^3+y_1^2+y_2^2y_5$\\
\hline
$\pp^4$&$   x_1x_2 + x_4^2 +x_5^2 $&$\pp^4$&$y_1y_2+y_4^2+y_5^2$\\
&$  x_3^3 +x_1^2x_4 +x_2^2x_5 + x_3x_4x_5$& quot: $1/8( 0, 2, 6, 3, 1 )$&$y_3^3 + y_1^2y_4 + y_2^2y_5 + y_3y_4y_5$\\
\hline  
$\pp^4$&$ x_1x_2 + x_3^2 +x_4^2 + x_5^2$&$\pp^4$&$ y_1y_2 +y_3^2 + y_4^2 + y_5^2$\\
&$   x_1^3 + x_2^3 + x_3x_4x_5$&quot: $1/2( 0, 0, 1, 1, 0 )$,&$  y_1^3 + y_2^3 + y_3y_4y_5$\\
&&    $1/6( 0, 4, 2, 5, 5 )$&\\
\end{tabular}
\caption{Delsarte quasismooth good pairs in $\pp^4$}
\label{tab_P11111}
\end{table}}

\begin{example} 
Let $Z=\pp(1,1,1,1,2)$ and consider the quasismooth family of K3 surfaces in $Z$ with equations (see Notation \ref{coef}):
\[  
\begin{cases}
g_1=x_1x_2x_3+ x_2^3  +x_4^3 + x_3x_5=0\\
g_2=x_1^3 + x_2x_3^2 + x_4x_5=0.
\end{cases}
\]
The dual family of K3 surfaces  is defined by the following equations in $Z^*=\pp( 7,5,15, 9,6 )$:
\[
\begin{cases}
g_1^*=y_1y_2y_3+y_5^2y_3+y_1^3y_5+y_4^3=0\\
g_2^*=y_3+y_2^3+y_4y_5=0.
\end{cases}
\]
Observe that the variable $y_3$ can be eliminated in the second equation and one obtains the family of hypersurfaces in $\pp(7,5,6,9)$ defined by
\[
y_2^3y_4^2+y_5y_4^3+y_1^3y_4+y_5^3+y_1y_2^4+y_1y_2y_4y_5=0.
\]
In \cite{Belcastro-paper} Belcastro considered K3 hypersurfaces in weighted projective spaces and studied lattice-theoretical mirror symmetry between them (see \cite{LPK3}).
The weighted projective space $\pp(5,6,7,9)$ appears in \cite[Table 3, case 84]{Belcastro-paper} and a general K3 hypersurface has Picard lattice of rank $18$ isometric to $U\oplus E_8\oplus A_8$, so that 
the general lattice-theoretical mirror surfaces should have Picard lattice of rank $2$, with quadratic form $w_{3,2}^{-1}$ (see \cite[Definition 2.3]{Belcastro-paper}). 
However, the mirror family is not in the list of \cite[Table 3]{Belcastro-paper}, i.e. it is not given by hypersurfaces in weighted projective spaces.
 We now show that the Picard lattice of the minimal resolution of a general K3 surface $X$ defined as before by $g_1=g_2=0$ in $\pp(1,1,1,1,2)$ contains a rank two sublattice with quadratic form $w_{3,2}^{-1}$,
thus giving the missing mirror in Belcastro list.

The projection $\psi:\pp(1,1,1,1,2)\to \pp^3, (x_1,\dots,x_5)\mapsto (x_1,\dots,x_4)$ induces an isomorphism between $X\backslash \{p\}$, where $p=(0,0,0,0,1)$, and $S\backslash L$, where $S$ is
 the quartic surface  of equation
\[
x_4(x_1x_2x_3+x_2^3+x_4^3)=x_3(x_1^3+x_2x_3^2)=0
\]
and $L$ is the line $x_3=x_4=0$. 
In fact $X$ has a singularity of type $A_1$ at $p$ 
and the resolution of $\psi$ induces an isomorphism between the minimal resolution $\tilde X$ of $X$ and  $S$.
Thus ${\rm Pic}(\tilde X)\cong {\rm Pic}(S)$ contains the sublattice with intersection matrix
$$
\begin{pmatrix}4&1\\1&-2\end{pmatrix},
$$
which corresponds to the quadratic form $w_{3,2}^{-1}$  \cite[Appendix A]{belcastro-thesis}. 
This process can be repeated to find the lattice mirror of other families of K3 surfaces in \cite[Table 3]{Belcastro-paper}.
\end{example}

\appendix
\section{Magma library}\label{magma}
We developed a library in MAGMA \cite{Magma} 
to explore aspects of the theory and to compute examples in the paper, either in terms of polytopes or of equations. 
These can be found in the webpage

\begin{center}    
\url{https://shorturl.at/vG48I}
\end{center}

As for quasismoothness, 
the  function \texttt{IsQsCI} checks quasismoothness of the complete intersection as in Definition \ref{qsci}.
In computations, we also use the necessary conditions of items \cite[Corollary 8.8, (ii) and (iii)]{F}  to eliminate cases which are not quasismooth.
The function \texttt{IsCY} checks the hypothesis of Theorem \ref{thmCY}.

Given a toric variety $Z$ and a partition of the vertices of the polar of the anticanonical polytope of $Z$, the function \texttt{GNP} checks if this is a generalized nef partition according to Definition \ref{gnp2}, whereas the function \texttt{AllGNP} finds all partitions defining a generalized nef partition for the variety $Z$.

If $\Pi(\Delta)$ is a generalized nef partition, the function \texttt{IsIrreducibleGNP} checks irreducibility of $\Pi(\Delta)$ as in Definition \ref{def-reducible}, whereas  \verb|Dual_partition| constructs the dual partition.
Given a pair of generalized nef partitions $(\Pi(\Delta^1), \Pi(\Delta^2))$, the function \texttt{Goodpairs} checks if they form a good pair of generalized nef partitions as in Definition \ref{goodpairs}. The functions \texttt{Eqs} and \verb|Eqs_to_pol| find the equations of a complete intersection asosciated to a good pair of generalized nef partitions and viceversa, respectively.

\bibliographystyle{plain}
\bibliography{Biblioqs}

\end{document}